\documentclass[11pt]{article}
\usepackage{latexsym,amssymb,amsmath,amsfonts,amsthm}
\usepackage{cite}
\usepackage{enumerate}
\usepackage{cases}
\usepackage{graphics}
\usepackage{graphicx}
\usepackage{mathrsfs}
\usepackage{color}
\usepackage{bm}
\usepackage{subfigure}
\newcommand{\R}{{\mat R}}

\newcommand{\N}{{\mat N}}
\newcommand{\C}{{\mat C}}
\newcommand{\E}{{\mat E}}

\newcommand{\be}{\begin{eqnarray}}
\newcommand{\ben}{\begin{eqnarray*}}
\newcommand{\en}{\end{eqnarray}}
\newcommand{\enn}{\end{eqnarray*}}

\newcommand{\dive}{{\rm div\,}}

\newcommand{\mat}{\mathbb}

\newcommand{\n}{\bm n}

\newcommand{\0}{\bm{0}}
\newtheorem{theorem}{Theorem}[section]
\newtheorem{lem}[theorem]{Lemma}

\newtheorem{remark}[theorem]{Remark}

\begin{document}
\renewcommand{\theequation}{\arabic{section}.\arabic{equation}}

\title{\bf
Analysis of a time-dependent fluid-solid interaction problem above a local rough surface}
\author{Changkun Wei\thanks{Academy of Mathematics and Systems Science, Chinese Academy of Sciences, Beijing 100190, China and School of Mathematical Sciences, University of Chinese Academy of  Sciences, Beijing 100049, China ({\tt weichangkun@amss.ac.cn})}
\and
Jiaqing Yang\thanks{The Corresponding Author, School of Mathematics and Statistics, Xi'an Jiaotong University,
Xi'an, Shaanxi, 710049, P. R. China ({\tt jiaq.yang@xjtu.edu.cn})}
}
\date{}

\maketitle


\begin{abstract}
This paper is concerned with the mathematical analysis of time-dependent fluid-solid interaction problem associated with a bounded elastic body immersed in a homogeneous air or fluid above a local rough surface. We reformulate the unbounded scattering problem into an equivalent initial-boundary value problem defined in a bounded domain by proposing a transparent boundary condition (TBC) on a hemisphere. Analyzing the reduced problem with Lax-Milgram lemma and abstract inversion theorem of Laplace transform, we prove the well-posedness and stability for the reduced problem. Moreover, an a priori estimate is established directly in the time domain for the acoustic wave and elastic displacement with using the energy method.

\vspace{.2in}
\end{abstract}

\section{Introduction}
\setcounter{equation}{0}

 The interaction between an elastic body and a compressible, inviscid fluid is generally referred to as a fluid-solid interaction which can be mathematically formulated as an initial-boundary value transmission problem. The study on this problem has been a subject of interest in both the mathematical and engineering community; see
 \cite{Donea1982,GLZ17,Hamdi1986,Hsiao1989,Hsiao1994,Hsiao17,LM95,MO95} and the references therein. However, most of the investigations study typical fluid-solid interaction problems confined to the time-harmonic setting. Two kinds of methods are usually used to
 prove the well-posedness of the scattering problems in the literature, including the boundary integral equation and variational techniques. For example, the existence
 of a solution was established in \cite{LM95}  with using the boundary integral equation method in the case of non-Jones frequencies for the elastic field, while the similar results were
 established in \cite{Hsiao1989,Hsiao2000} with using the various variational formulations.
Additionally, numerical solutions can be found in \cite{AH1988,everstine1990} and \cite{EH2010, HuRY2016} 
on the fluid-solid interaction problem with the boundary element or finite element. And the related inverse problems were also studied
numerically in \cite{hu2016,Yin2016} with the factorization method for imaging a periodic
interface or a bounded obstacle.

However, in most of real-world problems, the model setting not only depends on the space, but also
 depends on the time. This class of problems
have recently attracted much attention due to their capability of capturing wide-band signals and modeling more general material and nonlinearity \cite{ChenMonk2014,Jiao2002,Li2012,WW2014,Wang2012,Zhao2013}. 
Precisely, the mathematical analysis can be found in
\cite{chen2009,WW2014} for time-dependent scattering problems in the full acoustic wave cases,
and
\cite{chen2008,Gao2016,GL17,LWW15} in the
full electromagnetic wave cases, where bounded or unbounded scatterers were considered.
However, there exists few works on the time-dependent fluid-solid interaction
problems in the literature, especially for the unbounded rough surfaces. We here refer to \cite{Bao2018}
for a bounded elastic body in the two-dimensional case, and
\cite{GLZ17} for an unbounded layered structure in the three-dimensional case.
A key role in above works is played that a time domain transparent boundary condition (TBC) was proposed which can reduce the model problem in an unbounded domain into the bounded or infinite rectangular slab case. And
the reduced problem can be proved to be well-posed by the variational method in combination with
the Laplace transform and its inversion.
Moreover, a different technique was also proposed in \cite{Hsiao17} by the coupling of the boundary
integral equations and the variational formulation, in order to obtain the well-posedness of the scattering
problem.
Furthermore, several numerical studies were provided in \cite{EA91,FKW06,Soares2006} based on the mathematical analysis in \cite{Hsiao17} or the full boundary integral formulations.

In this paper, we intend to study the well-posedness and stability of the time-dependent fluid-solid interaction problem in three dimensions associated with a bounded elastic body embedded in the upper half-space with a local rough surface. Part of this work is motivated by the work \cite{LWW15} on the scattering by a three-dimensional open cavity governing by Maxwell's equations in time domain. A time
domain Transparent boundary condition (TBC) is introduced, defined on a hemisphere, which reduces
the model problem in the unbounded domain into the bounded one. The well-posedness and
stability of the reduced problem can be thus proved by the variational method in combination with Laplace transform and its inversion. Moreover, an a priori stability estimate of the solution is also obtained by using the energy method, which can provide an explicit
dependence on the time variable.

The outline of this paper is as follows. In section 2, we formulate our problem and put forward an exact transparent boundary condition (TBC) to reformulate the unbounded scattering problem into an equivalent  initial-boundary value problem in a bounded domain. In section 3, we study the well-posedness and stability for the reduced problem by the variation method, and also provide an a priori estimate for the acoustic field and elastic displacement with using the energy method.

\section{Problem formulation}\label{sec2}
\setcounter{equation}{0}
In this section, we formulate our problem by introducing model equations for the acoustic and elastic waves, present a transmission condition for the fluid-solid interaction
. In addition, an exact time-domain transparent boundary condition (TBC) shall be introduced which can reformulate the scattering problem into an initial-boundary value problem in a bounded domain, and some further properties on TBC will be also presented.

\subsection{An initial-boundary value transmission problem}
We consider a time-dependent fluid-solid interaction problem, which can be formulated as follows: A plane incident acoustic wave propagates in a fluid domain above a local rough surface in which a bounded elastic body is immersed; see Figure $1$. The forward problem is to determine the scattered pressure and velocity fields in the fluid domain as well as the displacement field in the elastic domain at any time.
\begin{figure}[!htbp]\label{fig1}
\setcounter{subfigure}{0}
  \centering
  \includegraphics[width=2.5in]{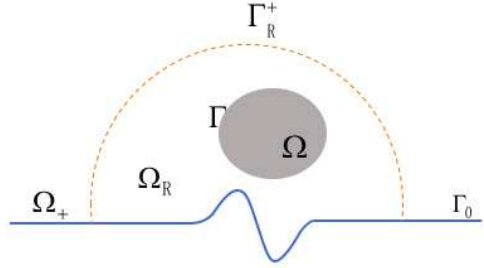}
\caption{Geometric configuration of the scattering problem}
\end{figure}

Throughout our paper, let $\Omega$ be the bounded homogeneous and isotropic elastic body with a Lipschitz boundary $\Gamma:=\partial\Omega$ in the unbounded domain $\Omega_+$, where $\Omega_+:=\{x\in\R^3:x_3>f_{\Gamma_0}(x_1,x_2),(x_1,x_2)\in \R^2\}$ with the boundary $\Gamma_0 := \partial\Omega_+ = \{x\in\R^3:x_3=f_{\Gamma_0}(x_1,x_2),(x_1,x_2)\in \R^2\}$ characterized by some smooth function $f_{\Gamma_0}\in C^2(\R^2)$ which has compact support in $\R^2$. We assume that the elastic body $\Omega$ is described by a constant mass density $\rho_e>0$ and let $\Omega^c = \Omega_+ \setminus \bar{\Omega}$ denote its exterior occupied by a compressible fluid with constant density $\rho_0>0$. Denote by $B^+_R:=\{x\in\Omega_+:|x|<R\}$ the half-ball and $\Gamma_R^+:=\{x\in\Omega_+:|x|=R\}$ the upper hemisphere, where $R>0$ is large enough such that $(\bar{\Omega}\cup supp f_{\Gamma_0})\subset B^+_R$. Let $\Omega_R=B^+_R\setminus\bar{\Omega}$ be the bounded region between $\Gamma_0$ ,$\Gamma$ and $\Gamma_R^+$. In what follows, we let $\n$ denote the unit outward normal vector on $\Gamma$ and $\Gamma_R^+$, directed into the exterior of the $\Omega$ and $B_R^+$, respectively. And, we define $\C_+:=\{s=s_1+is_2\in\C: s_1,s_2\in\R\;{\rm with\;}s_1>0\}.$\\

\textbf{Elastic domain}.\ In the elastic body $\Omega$, the elastic displacement $\bm u$ is governed by the linear elastrodynmic equation:
\begin{eqnarray}\label{2.1}
\rho_e\frac{\partial^{2}\bm u}{\partial t^{2}}-\Delta^*\bm u=0,\qquad (x,t)\in\Omega\times(0,\infty)
\end{eqnarray}
where $\Delta^*$ is the Lam\'{e} operator defined as follows
\begin{align*}
\Delta^*\bm u
:= \mu\Delta\bm u+(\lambda+\mu)\nabla\dive\bm u
 = \dive\bm\sigma(\bm u).
\end{align*}
In above,
$\bm \sigma(\bm u)$ and $\bm \varepsilon(\bm u)$ are called stress and strain tensors respectively, which are given by
$$\bm \sigma(\bm u)=(\lambda \dive\bm u)\bm{I}+2\mu\bm\varepsilon(\bm u)\qquad\text{and}\qquad   \bm\varepsilon(\bm u)=\frac{1}{2}(\nabla\bm u+(\nabla\bm u)^{T}).$$
Furthermore, Lam\'{e} constants $\lambda$ and $\mu$ are assumed to satisfy the condition that $\mu \geq 0$ and $3\lambda + 2\mu \geq 0.$ \\

\textbf {Fluid domain}.\ In the unbounded irrotational fluid domain $\Omega^c$, the scattered pressure $p$ and velocity fields $v$ can be represented by the velocity potential
$\varphi:=\varphi(x,t)$ as follows (cf. \cite{Hsiao17})
$$v=-\nabla\varphi,\quad\text{and}\quad p=\rho_0\frac{\partial\varphi}{\partial t},$$
where $\varphi$ is governed by the linear wave equation
\begin{eqnarray}\label{2.2}
\frac{\partial^{2}\varphi}{\partial t^{2}}-c^2\Delta\varphi=0\quad {\rm in}\;\; \Omega^c\times(0,\infty),
\end{eqnarray}
with the constant sound speed $c$.\par
Consider an incoming plane wave of the form
$$
\varphi^{inc}(x,t)=f(x\cdot d-ct),\quad x\in\Omega^c,\;t>0
$$
with a smooth function $f$ which is assumed to be of $C^k$-class $(k\geq3)$. Here, $d=(\sin\phi \cos\theta,\sin\phi \sin\theta,\cos\phi)\in S^2_-:=S^2\cap\R^3_-$ is the incident direction, $(\theta,\phi)$ are Euler angles with $\pi/2<\phi<\pi$ and $0<\theta<2\pi$, and $S^2=\{x\in\R^3:\ |x|=1\}$ is the unit sphere. It is easily checked that $\varphi^{inc}(x,t)$ satisfies acoustic wave equation (\ref{2.2}).
Moreover, we assume that the total field $\varphi=\varphi^{inc}+\varphi^{ref}+\varphi^{sc}$ vanishes on the surface $\Gamma_0$ known as Dirichlet boundary condition:
\begin{eqnarray}\label{2.3}
\varphi=\varphi^{inc}+\varphi^{ref}+\varphi^{sc}=0\quad{\rm on}\;\;\Gamma_0,
\end{eqnarray}
where $\varphi^{ref}(x,t): =-f(x\cdot d'-ct)$ is the reflected wave of $\varphi^{inc}$ by the infinite plane $\{x_3=0\}$ with $d'=(\sin\phi \cos\theta,\sin\phi \sin\theta,-\cos\phi)$, and $\varphi^{sc}$ is the scattered wave which satisfies the sommerfeld radiation condition
\begin{align}\label{2.4}
\partial_r\varphi^{sc}+\frac{1}{c}\partial_t\varphi^{sc}=o(\frac{1}{r}),\quad \text{as}\;\; r=|x|\rightarrow \infty,t>0.
\end{align}
Let $\varphi^{0}:=\varphi^{inc}+\varphi^{ref}$, and it holds that $\varphi^0$ has compact support on $\Gamma_0$ since $\Gamma_0$ is just a local perturbation of the infinite plane $\{x_3=0\}$.\par
In addition, the elasticity medium and fluid medium
are coupled in two distinct ways resulting in two kinds of interface conditions on $\Gamma$ \cite{Hsiao17}:\\
(i):\ kinematic interface condition
\begin{eqnarray}\label{2.5}
\frac{\partial \bm u}{\partial t}\cdot\n=-\frac{\partial \varphi}{\partial\n},
\end{eqnarray}
(ii):\ dynamic interface condition
\begin{eqnarray}\label{2.6}
\bm{\sigma(u)n}=-\rho_0\frac{\partial\varphi}{\partial t}\n.
\end{eqnarray}

The time-dependent scattering problem can be now modelled by combining (\ref{2.1}) for the elastic displacement field $\bm{u}$ and (\ref{2.2}) for velocity potential $\varphi$ together with the interface conditions (\ref{2.5})-(\ref{2.6}) as well as the homogeneous initial conditions
\begin{eqnarray}\label{2.7}
\bm u(x,0)=\frac{\partial \bm u(x,0)}{\partial t}=\0,\ x \in \Omega\quad \text{and} \quad  \varphi(x,0)=\frac{\partial \varphi(x,0)}{\partial t}=0,\ x\in \Omega^c,
\end{eqnarray}
which can be formulated as the following PDE-system:
\begin{equation}\label{2.8}
\begin{cases}
\rho_e\frac{\partial^{2}\bm u}{\partial t^{2}}-\Delta^*\bm u=0,& {\rm in}\;\;\Omega, t>0 \\
\frac{\partial^{2}\varphi}{\partial t^{2}}-c^2\Delta\varphi=0,& {\rm in}\;\;\Omega^c, t>0\\
\bm u(x,0)=\frac{\partial \bm u(x,0)}{\partial t}=\textbf 0,& {\rm in}\;\; \Omega\\
\varphi(x,0)=\frac{\partial \varphi(x,0)}{\partial t}=0,& {\rm in}\;\; \Omega^c\\
\bm{\sigma(u)\n}=-\rho_0\frac{\partial \varphi}{\partial t}\n,& {\rm on}\;\;\Gamma, t>0 \\
\frac{\partial \bm u}{\partial t}\cdot\n=-\frac{\partial \varphi}{\partial \n},& {\rm on}\;\; \Gamma, t>0\\
\varphi=\varphi^{inc}+\varphi^{ref}+\varphi^{sc}=0,& {\rm on}\;\;\Gamma_0, t>0\\
\partial_r\varphi^{sc}+\frac{1}{c}\partial_t\varphi^{sc}=o(\frac{1}{r}).&\text{as}\;\;r=|x|\rightarrow \infty,t>0.
\end{cases}
\end{equation}
\subsection{Transparent boundary condition on a hemisphere.} In this subsection, we aim to propose
 a transparent boundary condition (TBC) on a hemiphere which can reduce the model problem (\ref{2.8})
 into the case in the bounded domain $B_R^+$ with the sommerfeld radiation condition (\ref{2.4})
 replaced by
 \begin{eqnarray}\label{2.9}
\partial_{\n} \varphi =\mathscr{T}[\varphi]+\rho,\quad {\rm on}\;\;\Gamma_R^+, t>0,
\end{eqnarray}
then we obtain the equivalent reduced sysytem
\begin{equation}\label{reduced}
\begin{cases}
\rho_e\frac{\partial^{2}\bm u}{\partial t^{2}}-\Delta^*\bm u=0,& {\rm in}\;\;\Omega, t>0 \\
\frac{\partial^{2}\varphi}{\partial t^{2}}-c^2\Delta\varphi=0,& {\rm in}\;\;\Omega_R, t>0\\
\bm u(x,0)=\frac{\partial \bm u(x,0)}{\partial t}=\textbf 0,& {\rm in}\;\;  \Omega\\
\varphi(x,0)=\frac{\partial \varphi(x,0)}{\partial t}=0,& {\rm in}\;\;  \Omega_R\\
\bm{\sigma(u)n}=-\rho_0\frac{\partial \varphi}{\partial t}\n,& {\rm on}\;\;\Gamma, t>0 \\
\frac{\partial \bm u}{\partial t}\cdot\n=-\frac{\partial \varphi}{\partial\n},& {\rm on}\;\;\Gamma, t>0\\
\varphi=0,&{\rm on}\;\;\Gamma_0, t>0\\
\partial_{\n} \varphi=\mathscr{T}[\varphi]+\rho,& {\rm on}\;\;\Gamma_R^+, t>0.
\end{cases}
\end{equation}
Here, (\ref{2.9}) is called the TBC in time domain and $\rho:=\partial_{\n} \varphi^0-\mathscr{T}[\varphi^0].$\par

In what follows,  we will derive the explicit representation of the operator $\mathscr{T}$ and then present some basic properties. To this end, we take the Laplace transforms
$\check{\varphi}(x,s)=\mathscr{L}(\varphi)(x,s)$ and $\check{\bm u}(x,s)=\mathscr{L}(\bm u)(x,s)$ of $\varphi(x,t)$ and $\bm u(x,t)$, respectively, in (\ref{reduced}) with respect to $t$ (see related definitions on the Laplace transform in Appendix A). Then (\ref{2.8}) can be reduced to the time harmonic
fluid-solid interaction problem in s-domain:
\begin{equation}\label{2.10}
\begin{cases}
\Delta^*\check{\bm u}-\rho_es^2\check{\bm u}=0,& {\rm in}\;\; \Omega\\
\Delta\check{\varphi}-\frac{s^{2}}{c^{2}}\check{\varphi}=0,& {\rm in}\;\; \Omega_R\\
\bm{\sigma(\check u)n}=-\rho_0s\check{\varphi}\n,& {\rm on}\;\; \Gamma\\
s\check{\bm u}\cdot \n=-\frac{\partial \check{\varphi}}{\partial \n},& {\rm on}\;\; \Gamma\\
\check{\varphi}=0,& {\rm on}\;\;\Gamma_0\\
\partial_{\n} \check{\varphi}=\mathscr{B}[\check{\varphi}]+\check{\rho},& {\rm on}\;\; \Gamma_R^+
\end{cases}
\end{equation}
where $s\in\C_+$, and $\mathscr{B}$ is the Dirichlet-to-Neumann (DtN) operator in s-domain which will be defined by spherical harmonics expansion later. It follows from (\ref{2.9}) that
$\mathscr{T}=\mathscr{L}^{-1}\circ\mathscr{B}\circ\mathscr{L}$.\par
Let $Y_n^m(\hat{x}), m=-n,...n, n=1,2,...,$ be the spherical harmonics which forms a complete orthonormal basis on $L^2(S^2)$. It then follows from \cite[Lemma 3.1]{LWZ12} that
$X_n^m(\theta,\varphi): =\frac{\sqrt{2}}{R}Y_n^m(\theta,\varphi)$, $|m|\leq n, m+n = odd, n\in\N$, form a complete orthonormal system on $L^2(\Gamma_R^+)$. Noting that $\check{\varphi}^0=\mathscr{L}(\varphi^0)$ satisfies the Helmholtz equation $\Delta\check{\varphi}^0-\frac{s^{2}}{c^{2}}\check{\varphi}^0=0$, so it follows from the radiation condition (\ref{2.4}) that the scattered field $\check{\varphi}^{sc}=\mathscr{L}(\varphi^{sc})$ has the following expansion
\begin{eqnarray}\label{2.11}
\check{\varphi}^{sc}=\sum_{|m|\leq n}^{odd}\alpha_n^m h_n^{(1)}(\frac{is}{c}r)X_n^m(\theta,\varphi),\ {\rm in}\;\; \Omega_+\setminus B_R^+,
\end{eqnarray}
where $\sum_{|m|\leq n}^{odd}\omega_n^m:=\sum_{n=1}^{\infty}\sum_{\mathop{{m=-n}}\limits_{m+n=odd}}^{n}\omega_n^m$ for any sequence $\{\omega_n^m\}$. Therefore, one has
\begin{eqnarray}\label{2.12}
\partial_{\n}\check{\varphi}^{sc}|_{\Gamma_R^+}=\sum_{|m|\leq n}^{odd}a_n^m\gamma_nX_n^m(\theta,\varphi),
\end{eqnarray}
where $a_n^m:=\alpha_n^m h_n^{(1)}(\frac{is}{c}R)$ and $\gamma_n:=\frac{is}{c} {h_n^{(1)}}^{'}(\frac{is}{c}R)/h_n^{(1)}(\frac{is}{c}R)$ makes sense since $h_n^{(1)}(z)$ has no zeros (cf. \cite[Appendix C]{LWW15}). This leads to that the DtN operator $\mathscr{B}:H^{1/2}(\Gamma_R^+)\rightarrow H^{-1/2}(\Gamma_R^+)$ can be defined by
\begin{eqnarray}\label{2.13}
\mathscr{B}\omega=\sum_{|m|\leq n}^{odd}\gamma_n\beta_n^m X_n^m(\theta,\varphi)\quad for\;\;\omega=\sum_{|m|\leq n}^{odd}\beta_n^m X_n^m.
\end{eqnarray}
which is clearly proved to be bounded from $H^{1/2}(\Gamma_R^+)$ to $H^{-1/2}(\Gamma_R^+)$ with using the asymptotic behavior of $h_n^{(1)}$ as $n\to\infty$.

By (\ref{2.12}), it is now verified that the total field $\check{\varphi}$ satisfies the TBC in $s$-domain in
(\ref{2.10}) with $\check{\rho}:=\partial_{\n}\check{\varphi}^0-\mathscr{B}\check{\varphi}^0$.
Taking the inverse Laplace transform of the last equality in (\ref{2.10}) yields (\ref{2.9}) in the time-domain.

\begin{lem}\label{lem2.1}
For any $\omega\in H^{1/2}(\Gamma_R^+)$, it holds
$$-Re\langle s^{-1}\mathscr{B}\omega,\omega\rangle_{\Gamma_R^+}\geq 0.$$
\end{lem}

\begin{proof}
For $\omega=\sum_{|m|\leq n}^{odd}\omega_n^m X_n^m \in H^{1/2}(\Gamma_R^+)$,
it is deduced from (\ref{2.13}) that
$$-Re\langle s^{-1}\mathscr{B}\omega,\omega\rangle_{\Gamma_R^+}=-\sum_{|m|\leq n}^{odd}\frac{i}{c}\frac{{h_n^{(1)}}^{'}(\frac{is}{c}R)}{h_n^{(1)}(\frac{is}{c}R)}|\omega_n^m|^2.$$
By \cite[Lemma 2.3 and Lemma 2.4]{chen2008} with $s=s_1+i s_2$ ($s_1>0$) and $r_n^{(1)}(z)=z{h_n^{(1)}}^{'}(z)/h_n^{(1)}(z)$, we can easily get
\begin{eqnarray*}
Re\left\{\bar{s}[1+r_n^{(1)}(\frac{isR}{c})]\right\}\leq 0,\quad or \quad
Re\left\{s^{-1}[1+r_n^{(1)}(\frac{isR}{c})]\right\}\leq 0,
\end{eqnarray*}
which implies that
$$Re\left\{\frac{iR}{c}\frac{{h_n^{(1)}}^{'}(\frac{is}{c}R)}{h_n^{(1)}(\frac{is}{c}R)}\right\}\leq -Res^{-1}=-\frac{s_1}{|s|^2}\leq 0.$$
This completes the proof.
\end{proof}

\section{The reduced problem}\label{sec3}
\setcounter{equation}{0}

In this section, we shall obtain the main result of this paper including the well-posedness and stability of the reduced problem (\ref{2.10}) and an a priori estimate of the solution.
\subsection{Well-posedness in s-domain}
Consider the reduced boundary value problem
\begin{subnumcases}{} \label{3.1a}
\Delta^*\check{\bm u}-\rho_es^2\check{\bm u}=0 & ${\rm in}\;\; \Omega$ \\ \label{3.1b}
\Delta\check{\varphi}-\frac{s^{2}}{c^{2}}\check{\varphi}=0 & ${\rm in}\;\; \Omega_R$\\\label{3.1c}
\bm{\sigma(\check u)n}=-\rho_0s\check{\varphi}\n & ${\rm on}\;\; \Gamma$\\\label{3.1d}
s\check{\bm u}\cdot \n=-\frac{\partial \check{\varphi}}{\partial \n} & ${\rm on}\;\; \Gamma$\\ \label{3.1e}
\check{\varphi}=0 & ${\rm on}\;\;\Gamma_0$\\\label{3.1f}
\partial_{\n} \check{\varphi}=\mathscr{B}[\check{\varphi}]+\check{\rho}& ${\rm on}\;\; \Gamma_R^+$
\end{subnumcases}
in the product space $H:=\tilde{H_0^1}(\Omega_R)\times H^1(\Omega)^3$ with $\tilde{H_0^1}(\Omega_R):=\left\{u\in H^1(\Omega_R):u=0\;{\rm on}\;\Gamma_0\right\}$ which is clearly a closed subspace of $H^1(\Omega_R)$ and thereby a Hilbert space. We shall prove that Problem (\ref{3.1a})-(\ref{3.1f}) is well-posed in $H$ by the Lax-Milgram lemma. To this end, we first derive the variation formulation of (\ref{3.1a})-(\ref{3.1f}) by multiplying (\ref{3.1b}) and (\ref{3.1a})  the complex conjugates of a pair of test functions ($\check{\psi},\check{\bm{v}})\in H$, respectively, and applying Green's and Betti's formulas, transmission conditions (\ref{3.1c})-(\ref{3.1d}), and TBC (\ref{3.1f}). That is, we find
a solution $(\check{\varphi},\check{\bm{u}})\in H$ such that
\begin{equation}\label{3.2}
a\left((\check{\varphi},\check{\bm{u}}),(\check{\psi},\check{\bm{v}})\right)=\rho_0\int_{\Gamma_R^+}s^{-1}\check{\rho}\cdot\bar{\check{\psi}}d\gamma\quad {\rm for\;all}\;\; (\check{\psi},\check{\bm{v}})\in H,
\end{equation}
where the sesquilinear form $a(\cdot,\cdot)$ is defined as
\begin{equation}\label{3.3}
\begin{aligned}
a\left((\check{\varphi},\check{\bm{u}}),(\check{\psi},\check{\bm{v}})\right)&=\int_{\Omega_R}\rho_0(s^{-1}\nabla\check{\varphi}\cdot\nabla\bar{\check{\psi}}+\frac{s}{c^2}\check{\varphi}\cdot\bar{\check{\psi}})dx\\
&+\int_\Omega\left[s^{-1}\mu(\nabla\check{\bm{u}}:\nabla\bar{\check{\bm{v}}})+s^{-1}(\lambda+\mu)(\nabla\cdot\check{\bm{u}})(\nabla\cdot\bar{\check{\bm{v}}})+\rho_es\check{\bm{u}}\cdot\bar{\check{\bm{v}}}\right]dx\\
&-\rho_0\int_{\Gamma_R^+}s^{-1}\mathscr{B}[\check{\varphi}]\cdot\bar{\check{\psi}}d\gamma-\rho_0\int_\Gamma\check{\bm{u}}\cdot \n\bar{\check{\psi}}d\gamma+\rho_0\int_\Gamma\check{\varphi}\n\cdot\bar{\check{\bm{v}}}d\gamma,
\end{aligned}
\end{equation}
with $A:B=tr(AB^{T})$ the Frobenius inner product of square matrices $A$ and $B$.

Hereafter, we claim that the expressions $a\lesssim b$ or $a\gtrsim b$ means $a\leq Cb$ or $a\geq Cb$, respectively, where $C$ denotes a generic positive constant which does not depend on any function and important parameters in our model.
\begin{theorem}\label{thm3.1}
The variational problem (\ref{3.2}) has a unique solution $(\check{\varphi},\check{\bm{u}})\in H$ satisfying
\begin{align}
&\Vert\nabla\check{\varphi}\Vert_{L^2(\Omega_R)^3}+\Vert s\check{\varphi}\Vert_{L^2(\Omega_R)}\lesssim \frac{(1+|s|)^2}{s_1}\Vert\check{\tilde{\varphi}}^{inc}\Vert_{H^{1/2}(\Gamma_R^+)},\label{3.4}\\
&\Vert\nabla\check{\bm{u}}\Vert_{F(\Omega)}+\Vert \nabla\cdot\check{\bm{u}}\Vert_{L^2(\Omega)}+\Vert s\check{\bm{u}}\Vert_{{L^2(\Omega)}^3}\lesssim \frac{(1+|s|)^2}{s_1}\Vert\check{\tilde{\varphi}}^{inc}\Vert_{H^{1/2}(\Gamma_R^+)}.\label{3.5}
\end{align}
where $s\in\C_+$,  $\tilde{\varphi}^{inc}:=(\varphi^{inc},\varphi^{ref})$ with $H^{1/2}$-norm $\Vert\check{\tilde{\varphi}}^{inc}\Vert_{H^{1/2}(\Gamma_R^+)}$ defined in (\ref{incident}) and $\Vert\nabla\check{\bm{u}}\Vert_{F(\Omega)}$ is the Frobenius norm defined by
\begin{eqnarray*}
    \Vert\nabla\check{\bm{u}}\Vert_{F(\Omega)}:=\Big(\sum_{j=1}^3\int_{\Omega}|\nabla\check{u}_j|^2dx\Big)^{1/2}.
\end{eqnarray*}
\end{theorem}
\begin{proof}
i) $a(\cdot,\cdot)$ is continuous. Indeed, by Cauchy-Schwartz inequality, the boundedness of $\mathscr{B}$ and the trace theorem, it follows that
\begin{equation}\label{3.6}
\begin{aligned}
|a\left((\check{\varphi},\check{\bm{u}}),(\check{\psi},\check{\bm{v}})\right)|
&\lesssim|s|^{-1}\Vert\nabla\check{\varphi}\Vert_{L^2(\Omega_R)^3}\cdot\Vert\nabla\check{\psi}\Vert_{L^2(\Omega_R)^3}+c^{-2}|s|\Vert\check{\varphi}\Vert_{L^2(\Omega_R)}\cdot\Vert\check{\psi}\Vert_{L^2(\Omega_R)}\\
&\quad+|s|^{-1}\Vert\mathscr{B}[\check{\varphi}]\Vert_{H^{-1/2}(\Gamma_R^+)}\cdot\Vert\check{\psi}\Vert_{H^{1/2}(\Gamma_R^+)}
+\Vert\check{\bm{u}}\cdot n\Vert_{L^2(\Gamma)}\Vert\check{\psi}\Vert_{L^2(\Gamma)}\\
&\quad+|s|^{-1}\Vert\nabla\check{\bm{u}}\Vert_{F(\Omega)}\Vert\nabla\check{\bm{v}}\Vert_{F(\Omega)}+|s|^{-1}\Vert\nabla\cdot\check{\bm{u}}\Vert_{L^2(\Omega)}\Vert \nabla\cdot\check{\bm{v}}\Vert_{L^2(\Omega)}\\
&\quad+|s|\Vert\check{\bm{u}}\Vert_{{L^2(\Omega)}^3}\Vert\check{\bm{v}}\Vert_{{L^2(\Omega)}^3}+\Vert\check{\varphi}\Vert_{{L^2(\Gamma)}}\Vert\check{\bm{v}}\cdot n\Vert_{L^2(\Gamma)}\\
&\lesssim\Vert\check{\varphi}\Vert_{H^1(\Omega_R)}\cdot\Vert\check{\psi}\Vert_{H^1(\Omega_R)}+\Vert\check{\varphi}\Vert_{H^1(\Omega_R)}\cdot\Vert\check{\psi}\Vert_{H^1(\Omega_R)}\\
&\quad+\Vert\check{\varphi}\Vert_{H^{1/2}(\Gamma_R^+)}\cdot\Vert\check{\psi}\Vert_{H^{1/2}(\Gamma_R^+)}+\Vert\check{\bm{u}}\Vert_{H^{1/2}(\Gamma)^3}\Vert\check{\psi}\Vert_{H^{1/2}(\Gamma)}\\
&\quad+\Vert\check{\bm{u}}\Vert_{H^1(\Omega)^3}\Vert\check{\bm{v}}\Vert_{H^1(\Omega)^3}+\Vert\check{\bm{u}}\Vert_{H^1(\Omega)^3}\Vert\check{\bm{v}}\Vert_{H^1(\Omega)^3}+\Vert\check{\varphi}\Vert_{H^{1/2}(\Gamma)}\Vert\check{\bm{v}}\Vert_{H^{1/2}(\Gamma)^3}\\
&\lesssim\Vert\check{\varphi}\Vert_{H^1(\Omega_R)}\cdot\Vert\check{\psi}\Vert_{H^1(\Omega_R)}+\Vert\check{\bm{u}}\Vert_{H^1(\Omega)^3}\Vert\check{\psi}\Vert_{H^1(\Omega_R)}\\
&\quad +\Vert\check{\bm{u}}\Vert_{H^1(\Omega)^3}\Vert\check{\bm{v}}\Vert_{H^1(\Omega)^3}+
\Vert\check{\varphi}\Vert_{H^1(\Omega_R)}\Vert\check{\bm{v}}\Vert_{H^1(\Omega)^3}
\end{aligned}
\end{equation}
which yields that $a(\cdot,\cdot)$ is continuous in the product space $H\times H$.\par
ii) $a(\cdot,\cdot)$ is strictly coercive. For $(\check{\psi},\check{\bm{v}}): = (\check{\varphi},\check{\bm{u}})$,
(\ref{3.3}) becomes
\begin{equation}\label{3.7}
\begin{aligned}
a\left((\check{\varphi},\check{\bm{u}}),(\check{\varphi},\check{\bm{u}})\right)&=\int_{\Omega_R}\rho_0(s^{-1}|\nabla\check{\varphi}|^2+\frac{s}{c^2}|\check{\varphi}|^2)dx\\
&+\int_\Omega\left[s^{-1}\mu(\nabla\check{\bm{u}}:\nabla\bar{\check{\bm{u}}})+s^{-1}(\lambda+\mu)(\nabla\cdot\check{\bm{u}})(\nabla\cdot\bar{\check{\bm{u}}})+\rho_es|\check{\bm{u}}|^2\right]dx\\
&-\rho_0\int_{\Gamma_R^+}s^{-1}\mathscr{B}[\check{\varphi}]\cdot\bar{\check{\varphi}}d\gamma+\rho_0\int_\Gamma(\check{\varphi}\n\cdot\bar{\check{\bm{u}}}-\check{\bm{u}}\cdot \n\bar{\check{\varphi}})d\gamma.
\end{aligned}
\end{equation}
Taking the real part of (3.7) and using Lemma \ref{lem2.1}, we get
\begin{equation}\label{3.8}
\begin{aligned}
Re\{a\left((\check{\varphi},\check{\bm{u}}),(\check{\varphi},\check{\bm{u}})\right)\}&\gtrsim\frac{s_1}{|s|^2}\left(\Vert\nabla\check{\varphi}\Vert_{L^2(\Omega_R)^3}^2
+\Vert s\check{\varphi}\Vert_{L^2(\Omega_R)}^2\right)\\
&+\frac{s_1}{|s|^2}\left(\Vert\nabla\check{\bm{u}}\Vert_{F(\Omega)}^2+\Vert\nabla\cdot\check{\bm{u}}\Vert_{{L^2(\Omega)}}^2+\Vert s\check{\bm{u}}\Vert_{{L^2(\Omega)}^3}^2\right).
\end{aligned}
\end{equation}
It follows from the Lax-Milgram lemma that the variational problem (\ref{3.2}) has a unique solution $(\check{\varphi},\check{\bm{u}})\in H$ for each $s\in\C_+$. 

Moreover, recalling $\varphi^0 = \varphi^{inc} + \varphi^{ref}$ which leads to 
$\partial_{\bm{ n}}\varphi^{0} = -\big(\frac{d}{c}\partial_t\varphi^{inc}+\frac{d'}{c}\partial_t\varphi^{ref}\big)\cdot x/|x|$ from the definitions on $\varphi^{inc}$ and $\varphi^{ref}$, we then have
\begin{align}\label{incident}
	\Vert\partial_{\bm{n}}\check{\varphi}^0\Vert_{H^{-1/2}(\Gamma_R^+)}\lesssim |s|(\Vert \check{\varphi}^{inc}\Vert_{H^{1/2}(\Gamma_R^+)}+\Vert \check{\varphi}^{ref}\Vert_{H^{1/2}(\Gamma_R^+)}):=|s|\cdot\Vert\check{\tilde{\varphi}}^{inc}\Vert_{H^{1/2}(\Gamma_R^+)}
\end{align}
 by using the Laplace transform. This, together with the Cauchy-Schwartz inequality, the trace theorem and the boundedness of $\mathscr{B}$ as well as the definition of  $\check{\rho}$, yields that 
\begin{equation}\label{3.9}
\begin{aligned}
|a\left((\check{\varphi},\check{\bm{u}}),(\check{\varphi},\check{\bm{u}})\right)|&\lesssim|\rho_0\int_{\Gamma_R^+}s^{-1}\check{\rho}\cdot\bar{\check{\varphi}}d\gamma|\\
&\lesssim\frac{1}{|s|^2}\Vert\check{\rho}\Vert_{H^{-1/2}(\Gamma_R^+)}\Vert s\check{\varphi}\Vert_{H^{1/2}(\Gamma_R^+)}\\
&\lesssim\frac{1}{|s|^2}\Vert\partial_{\n}\check{\varphi}^0-\mathscr{B}[\check{\varphi}^0]\Vert_{H^{-1/2}(\Gamma_R^+)}\Vert s\check{\varphi}\Vert_{H^1(\Omega_R)}\\
&\lesssim\frac{1+|s|}{|s|^2}\Vert\check{\tilde{\varphi}}^{inc}\Vert_{H^{1/2}(\Gamma_R^+)}\left((1+|s|^2)(\Vert\nabla\check{\varphi}\Vert_{L^2(\Omega_R)^3}^2
+\Vert s\check{\varphi}\Vert_{L^2(\Omega_R)}^2)\right)^{\frac{1}{2}}\\
&\lesssim\frac{(1+|s|)^2}{|s|^2}\Vert\check{\tilde{\varphi}}^{inc}\Vert_{H^{1/2}(\Gamma_R^+)}\left(\Vert\nabla\check{\varphi}\Vert_{L^2(\Omega_R)^3}^2
+\Vert s\check{\varphi}\Vert_{L^2(\Omega_R)}^2\right)^{\frac{1}{2}}.
\end{aligned}
\end{equation}
Combining (\ref{3.8})-(\ref{3.9}) with the Cauchy-Schwartz inequality again, we get
\begin{align*}
	\Vert\nabla\check{\varphi}\Vert_{L^2(\Omega_R)^3}+\Vert s\check{\varphi}\Vert_{L^2(\Omega_R)}
	\lesssim\left(\Vert\nabla\check{\varphi}\Vert_{L^2(\Omega_R)^3}^2+\Vert s\check{\varphi}\Vert_{L^2(\Omega_R)}^2\right)^{\frac{1}{2}}
	\lesssim\frac{(1+|s|)^2}{s_1}\Vert\check{\tilde{\varphi}}^{inc}\Vert_{H^{1/2}(\Gamma_R^+)}.
\end{align*}
Similar discussings applied to $\check{\bm{u}}$ yields
\begin{equation}\label{3.10}
\begin{aligned}
&\quad\Vert\nabla\check{\bm{u}}\Vert_{F(\Omega)}^2+\Vert\nabla\cdot\check{\bm{u}}\Vert_{{L^2(\Omega)}}^2+\Vert s\check{\bm{u}}\Vert_{{L^2(\Omega)}^3}^2\\
&\lesssim\Vert\nabla\check{\bm{u}}\Vert_{F(\Omega)}^2+\Vert\nabla\cdot\check{\bm{u}}\Vert_{{L^2(\Omega)}}^2+\Vert s\check{\bm{u}}\Vert_{{L^2(\Omega)}^3}^2+\Vert\nabla\check{\varphi}\Vert_{L^2(\Omega_R)^3}^2
+\Vert s\check{\varphi}\Vert_{L^2(\Omega_R)}^2\\
&\lesssim\frac{|s|^2}{s_1}|a\left((\check{\varphi},\check{\bm{u}}),(\check{\varphi},\check{\bm{u}})\right)|\\
&\lesssim\frac{|s|^2}{s_1}\cdot\frac{(1+|s|)^2}{|s|^2}\Vert\check{\tilde{\varphi}}^{inc}\Vert_{H^{1/2}(\Gamma_R^+)}\left(\Vert\nabla\check{\varphi}\Vert_{L^2(\Omega_R)^3}^2
+\Vert s\check{\varphi}\Vert_{L^2(\Omega_R)}^2\right)^{\frac{1}{2}}\\
&\lesssim\frac{(1+|s|)^2}{s_1}\Vert\check{\tilde{\varphi}}^{inc}\Vert_{H^{1/2}(\Gamma_R^+)}\Big(\Vert\nabla\check{\bm{u}}\Vert_{F(\Omega)}^2+\Vert\nabla\cdot\check{\bm{u}}\Vert_{{L^2(\Omega)}}^2+\Vert s\check{\bm{u}}\Vert_{{L^2(\Omega)}^3}^2\\
&\quad+\Vert\nabla\check{\varphi}\Vert_{L^2(\Omega_R)^3}^2
+\Vert s\check{\varphi}\Vert_{L^2(\Omega_R)}^2\Big)^\frac{1}{2}.
\end{aligned}
\end{equation}
Applying Cauchy-Schwartz inequality again, we have\\
\begin{equation}
\begin{aligned}
&\quad\Vert\nabla\check{\bm{u}}\Vert_{F(\Omega)}+\Vert \nabla\cdot\check{\bm{u}}\Vert_{L^2(\Omega)}+\Vert s\check{\bm{u}}\Vert_{{L^2(\Omega)}^3}\\
&\lesssim\left(\Vert\nabla\check{\bm{u}}\Vert_{F(\Omega)}^2+\Vert\nabla\cdot\check{\bm{u}}\Vert_{{L^2(\Omega)}}^2+\Vert s\check{\bm{u}}\Vert_{{L^2(\Omega)}^3}^2+\Vert\nabla\check{\varphi}\Vert_{L^2(\Omega_R)^3}^2
+\Vert s\check{\varphi}\Vert_{L^2(\Omega_R)}^2\right)^\frac{1}{2}\\
&\lesssim\frac{(1+|s|)^2}{s_1}\Vert\check{\tilde{\varphi}}^{inc}\Vert_{H^{1/2}(\Gamma_R^+)}.\notag
\end{aligned}
\end{equation}
This ends our proof.
\end{proof}
\subsection{Well-posedness in time domain.}
Recalling that
\begin{equation}\label{3.11}
\varphi^{inc}(x,\cdot)=f(x\cdot d-c\cdot)\in C^k(k\geq3)\ \text{with respect to}\;t\  for\ any\ x\in\Omega^c,
\end{equation}
we shall under this assumption show the well-posedness of (\ref{reduced}).
\begin{lem}\label{lem3.2}
Given $t\geq 0$ and $\omega\in L^2\left(0,t;H^{1/2}(\Gamma_R^+)\right)$ with initial value $\omega(\cdot,0)=0$, it holds that
$$Re\int_0^t\int_{\Gamma_R^+}\mathscr{T}[\omega]\cdot\partial_\tau\bar{\omega}d\gamma d\tau\leq 0.$$
\end{lem}
\begin{proof}
First, we extend $\omega$ by $0$ into $\tilde{\omega}$
with respect to $\tau$, that is $\tilde{\omega}=0$ outside the interval $[0,t].$ Using the Parseval identity (\ref{A.5}) and Lemma \ref{lem2.1}, we have
\begin{align}
&Re\int_0^te^{-2s_1\tau}\int_{\Gamma_R^+}\mathscr{T}[\omega]\cdot\partial_\tau\bar{\omega}d\gamma d\tau\notag\\
={}&Re\int_{\Gamma_R^+}\int_0^{\infty}e^{-2s_1\tau}\mathscr{T}[\tilde{\omega}]\cdot\partial_\tau\bar{\tilde{\omega}}d\tau d\gamma\notag\\
={}&\frac{1}{2\pi}\int_{-\infty}^{\infty}Re\langle\mathscr{B}[\check{\tilde{\omega}}],s\check{\tilde{\omega}}\rangle_{\Gamma_R^+}ds_2\notag\\
={}&\frac{1}{2\pi}\int_{-\infty}^{\infty}|s|^2Re\langle s^{-1}\mathscr{B}[\check{\tilde{\omega}}],\check{\tilde{\omega}}\rangle_{\Gamma_R^+}ds_2\notag\\
\leq{}&0.\notag
\end{align}
which completes the proof by taking $s_1\rightarrow 0.$
\end{proof}

\begin{lem}\label{lem3.3}
Given $t\geq 0$ and $\omega\in L^2\left(0,t;H^{1/2}(\Gamma_R^+)\right)$ with initial value $\omega(\cdot,0)=\partial_t\omega(\cdot,0)=0$, it holds that
$$Re\int_0^t\int_{\Gamma_R^+}\mathscr{T}[\partial_\tau\omega]\cdot\partial_\tau^2\bar{\omega}d\gamma d\tau\leq 0.$$
\end{lem}
\begin{proof}
The proof is similar to that in Lemma \ref{lem3.2} with $\omega$ replaced by $\partial_\tau\omega$.
So we here omit its detailed proof.
\end{proof}

\begin{theorem}
The reduced initial-boundary value problem (\ref{reduced}) has a unique solution $\left(\varphi(x,t),\bm{u}(x,t)\right)$
satisfying
\begin{align*}
&\varphi(x,t)\in L^2\big(0,T;\tilde{H^1_0}(\Omega_R)\big)\cap H^1\left(0,T;L^2(\Omega_R)\right),\\
&\bm{u}(x,t)\in L^2\left(0,T;H^1(\Omega)^3\right)\cap H^1\left(0,T;L^2(\Omega)^3\right),
\end{align*}
with the stability estimate
\begin{align}
&\max\limits_{t\in[0,T]}\left[(\Vert\partial_t\varphi\Vert_{L^2(\Omega_R)}+\Vert\nabla\partial_t\varphi\Vert_{L^2(\Omega_R)^3})+(\Vert\partial_t\bm{u}\Vert_{L^2(\Omega)^3}+\Vert\nabla\cdot\bm{u}\Vert_{L^2(\Omega)}+\Vert\nabla\bm{u}\Vert_{F(\Omega)})\right]\notag\\
\lesssim{}&\Vert\rho\Vert_{L^1(0,T;H^{-1/2}(\Gamma_R^+))}+\max\limits_{t\in[0,T]}\Vert\partial_t\rho\Vert_{H^{-1/2}(\Gamma_R^+)}+\Vert\partial_t^2\rho\Vert_{L^1(0,T;H^{-1/2}(\Gamma_R^+))}\notag.
\end{align}
\end{theorem}
\begin{proof}
Simple calculations yields
\begin{align}
&\int_0^T(\Vert\nabla\varphi\Vert_{L^2(\Omega_R)^3}^2+\Vert\partial_t\varphi\Vert_{L^2(\Omega_R)}^2+\Vert\nabla\bm{u}\Vert_{F(\Omega)}^2+\Vert\partial_t\bm{u}\Vert_{L^2(\Omega)^3}^2)dt\notag\\
\leq{}&\int_0^Te^{-2s_1(t-T)}(  \Vert\nabla\varphi\Vert_{L^2(\Omega_R)^3}^2+\Vert\partial_t\varphi\Vert_{L^2(\Omega_R)}^2+\Vert\nabla\bm{u}\Vert_{F(\Omega)}^2+\Vert\partial_t\bm{u}\Vert_{L^2(\Omega)^3}^2)dt\notag\\
\lesssim{}&\int_0^{\infty}e^{-2s_1t}(  \Vert\nabla\varphi\Vert_{L^2(\Omega_R)^3}^2+\Vert\partial_t\varphi\Vert_{L^2(\Omega_R)}^2+\Vert\nabla\bm{u}\Vert_{F(\Omega)}^2+\Vert\partial_t\bm{u}\Vert_{L^2(\Omega)^3}^2)dt\notag,
\end{align}
we estimate the integral for simplicity
$$\int_0^{\infty}e^{-2s_1t}(  \Vert\nabla\varphi\Vert_{L^2(\Omega_R)^3}^2+\Vert\partial_t\varphi\Vert_{L^2(\Omega_R)}^2+\Vert\nabla\bm{u}\Vert_{F(\Omega)}^2+\Vert\partial_t\bm{u}\Vert_{L^2(\Omega)^3}^2)dt.$$
Recalling the reduced system in s-domain (\ref{2.10}), by the estimate (\ref{3.4}) and (\ref{3.5}) in Theorem \ref{thm3.1}, 
 it follows from \cite[Lemma 44.1]{treves1975} that $(\check{\varphi},\check{\bm{u}})$ are holomorphic functions of $s$ on the half plane $s_1>\gamma>0,$ where $\gamma$ is any positive constant. Hence we have from Lemma A.2 that the inverse Laplace transform of $\check{\varphi}$ and $\check{\bm{u}}$ exist and are supported in $[0,\infty].$

Denoted by $\varphi=\mathscr{L}^{-1}(\check{\varphi})$ and $\bm{u}=\mathscr{L}^{-1}(\check{\bm{u}})$, we have from the Parseval identity (\ref{A.5}) and (\ref{3.4})-(\ref{3.5}) and trace theorem that
\begin{equation*}
\begin{aligned}
&\int_0^{\infty}e^{-2s_1t}(\Vert\nabla\varphi\Vert_{L^2(\Omega_R)^3}^2+\Vert\partial_t\varphi\Vert_{L^2(\Omega_R)}^2
+\Vert\nabla\bm{u}\Vert_{F(\Omega)}^2+\Vert\partial_t\bm{u}\Vert_{L^2(\Omega)^3}^2)dt\\
={}&\frac{1}{2\pi}\int_{-\infty}^{\infty}(\Vert\nabla\check{\varphi}\Vert_{L^2(\Omega_R)^3}^2+\Vert s\check{\varphi}\Vert_{L^2(\Omega_R)}^2
+\Vert\nabla\check{\bm{u}}\Vert_{F(\Omega)}^2+\Vert s\check{\bm{u}}\Vert_{L^2(\Omega)^3}^2)ds_2\\
\lesssim{}&\int_{-\infty}^{\infty}\frac{(1+|s|)^4}{s_1^2}\Vert\check{\tilde{\varphi}}^{inc}\Vert_{H^1(\Omega_R)}^2ds_2\\
\lesssim{}&s_1^{-2}\int_{-\infty}^{\infty}(\Vert\mathscr{L}(\tilde{\varphi}^{inc})\Vert_{H^1(\Omega_R)}^2+\Vert \mathscr{L}(\partial_t\tilde{\varphi}^{inc})\Vert_{H^1(\Omega_R)}^2+\Vert \mathscr{L}(\partial_t^2\tilde{\varphi}^{inc})\Vert_{H^1(\Omega_R)}^2)ds_2\\
\lesssim{}&s_1^{-2}\int_0^{\infty}e^{-2s_1t}(\Vert\tilde{\varphi}^{inc}\Vert_{H^1(\Omega_R)}^2+\Vert \partial_t{\tilde{\varphi}^{inc}}\Vert_{H^1(\Omega_R)}^2+\Vert \partial_t^2{\tilde{\varphi}^{inc}}\Vert_{H^1(\Omega_R)}^2)dt.
\end{aligned}
\end{equation*}
which shows that
\begin{align*}
&\varphi(x,t)\in L^2\big(0,T;\tilde{H^1_0}(\Omega_R)\big)\cap H^1\left(0,T;L^2(\Omega_R)\right),\\
&\bm{u}(x,t)\in L^2\left(0,T;H^1(\Omega)^3\right)\cap H^1\left(0,T;L^2(\Omega)^3\right).
\end{align*}
Next, we shall prove the stability of solution with respect to the initial-boundary conditions. To this end,
we define the energy function
$$\varepsilon_1(t)=e_1(t)+e_2(t),\quad for\; t\in(0,T)$$
with
\begin{align*}
&e_1(t)=\Vert\frac{\rho_0^{1/2}}{c}\partial_t\varphi\Vert_{L^2(\Omega_R)}^2+\Vert\rho_0^{1/2}\nabla\varphi\Vert_{L^2(\Omega_R)}^2,\\
&e_2(t)=\Vert\rho_e^{1/2}\partial_t\bm{u}\Vert_{L^2(\Omega)^3}^2+\Vert(\lambda+\mu)^{1/2}\nabla\cdot\bm{u}\Vert_{L^2(\Omega)}^2+\Vert\mu^{1/2}\nabla\bm{u}\Vert_{F(\Omega)}^2.
\end{align*}
Note that $\varepsilon(\cdot)$ can be written as
\begin{equation}\label{3.12}
\varepsilon_1(t)-\varepsilon_1(0)=\int_0^t\varepsilon_1^{'}(\tau)d\tau=\int_0^t\Big(e_1^{'}(\tau)+e_2^{'}(\tau)\Big)d\tau.
\end{equation}
Using the system (\ref{reduced}) and integrating by parts, we get
\begin{equation}\label{3.13}
\begin{aligned}
\int_0^te_1^{'}(\tau)d\tau&=2\rho_0Re\int_0^t\int_{\Omega_R}\left(\frac{1}{c^2}\partial_\tau^2\varphi\cdot\partial_\tau\bar{\varphi}+\partial_\tau(\nabla\varphi)\cdot\nabla\bar{\varphi}\right)dxd\tau\\
&=2\rho_0Re\int_0^t\int_{\Omega_R}\left(\Delta\varphi\cdot\partial_\tau\bar{\varphi}+\partial_\tau(\nabla\varphi)\cdot\nabla\bar{\varphi}\right)dxd\tau\\
&=\rho_0\int_0^t\int_{\Omega_R}2Re\left(-\nabla\varphi\cdot\nabla(\partial_\tau\bar{\varphi})+\partial_\tau(\nabla\varphi)\cdot\nabla\bar{\varphi}\right)dxd\tau\\
&\quad+2Re\rho_0\int_0^t\int_{\Gamma_R^+}\partial_{\n}\varphi\cdot\partial_\tau\bar{\varphi}d\gamma d\tau-2Re\rho_0\int_0^t\int_{\Gamma}\partial_{\n}\varphi\cdot\partial_\tau\bar{\varphi}d\gamma d\tau\\
&=2Re\rho_0\int_0^t\int_{\Gamma_R^+}(\mathscr{T}[\varphi]+\rho)\partial_\tau\bar{\varphi}d\gamma d\tau+2Re\rho_0\int_0^t\int_{\Gamma}\frac{\partial{\bm{u}}}{\partial \tau}\cdot \n\partial_\tau\bar{\varphi}d\gamma d\tau.
\end{aligned}
\end{equation}
and
\begin{align}\label{3.14}
&\int_0^te_2^{'}(\tau)d\tau\notag\\
={}&2Re\int_0^t\int_{\Omega}\Big(\rho_e\partial_\tau^2\bm{u}\cdot\partial_\tau\bar{\bm{u}}+(\lambda+\mu)\partial_\tau(\nabla\cdot\bm{u})\cdot\nabla\cdot\bar{\bm{u}}+\mu\partial_\tau(\nabla\bm{u}):\nabla\bar{\bm{u}}\Big)dxd\tau\notag\\
={}&\int_0^t\int_{\Omega}2Re\Big(\Delta^{*}\bm{u}\cdot\partial_\tau\bar{\bm{u}}+(\lambda+\mu)\partial_\tau(\nabla\cdot\bm{u})\cdot\nabla\cdot\bar{\bm{u}}+\mu\partial_\tau(\nabla\bm{u}):\nabla\bar{\bm{u}}\Big)dxd\tau\notag\\
={}&\int_0^t\int_{\Omega}2Re\Big(-\mu\nabla\bm{u}:\nabla(\partial_\tau\bar{\bm{u}})-(\lambda+\mu)(\nabla\cdot\bm{u})(\nabla\cdot(\partial_\tau\bar{\bm{u}}))\notag\\
&+(\lambda+\mu)\partial_\tau(\nabla\cdot\bm{u})\cdot(\nabla\cdot\bar{\bm{u}})+\mu\partial_\tau(\nabla\bm{u}):\nabla\bar{\bm{u}}\Big)dxd\tau
+2Re\int_0^t\int_{\Gamma}\sigma(\check{\bm{u}})\n\cdot\partial_\tau\bar{\bm{u}}d\gamma d\tau\notag\\
={}&-2Re\int_0^t\int_{\Gamma}\rho_0\partial_\tau\varphi \n\cdot\partial_\tau\bar{\bm{u}}d\gamma d\tau.
\end{align}
Combining (\ref{3.12})-(\ref{3.14}) with $\varepsilon_1(0)=0,$ we have
$$\varepsilon_1(t)=2Re\rho_0\int_0^t\int_{\Gamma_R^+}\mathscr{T}[\varphi]\cdot\partial_\tau\bar{\varphi}d\gamma d\tau+2Re\rho_0\int_0^t\int_{\Gamma_R^+}\rho\cdot\partial_\tau\bar{\varphi}d\gamma d\tau.$$
Using Lemma \ref{lem3.2} and the trace theorem, we arrive at the following estimate
\begin{align}\label{3.15}
&\Vert\partial_t\varphi\Vert_{L^2(\Omega_R)}^2+\Vert\nabla\varphi\Vert_{L^2(\Omega_R)^3}^2\notag\\
\lesssim{}&\varepsilon_1(t)\leq{}2Re\rho_0\int_0^t\int_{\Gamma_R^+}\rho\cdot\partial_\tau\bar{\varphi}d\gamma d\tau\notag\\
\lesssim{}&\int_0^t\Vert\rho\Vert_{H^{-1/2}(\Gamma_R^+)}\cdot\Vert\partial_\tau\varphi\Vert_{H^{1}(\Omega_R)}d\tau\notag\\
\lesssim{}&\max\limits_{t\in[0,T]}\Vert\partial_t\varphi\Vert_{H^{1}(\Omega_R)}\Vert\rho\Vert_{L^1(0,T;H^{-1/2}(\Gamma_R^+))}\notag.\\
\leq{}&\epsilon\max\limits_{t\in[0,T]}(\Vert\partial_t\varphi\Vert_{L^2(\Omega_R)}^2+\Vert\nabla\partial_t\varphi\Vert_{L^2(\Omega_R)^3}^2)+\frac{1}{4\epsilon}\Vert\rho\Vert^2_{L^1(0,T;H^{-1/2}(\Gamma_R^+))}
\end{align}
where the $\epsilon-$inequality has been used in deriving the last inequality.\\
Noticing the right-hand side of (\ref{3.15}) contains the term $\nabla\partial_t\varphi$ which cannot be controlled by the left-hand side of (\ref{3.15}), we then need to
consider the following system
\begin{equation}\label{3.16}
\begin{cases}
\rho_e\partial_t^2(\partial_t\bm u)-\Delta^*(\partial_t\bm u)=0 & {\rm in}\;\; \Omega, t>0 \\
\partial_t^2(\partial_t\varphi)-c^2\Delta(\partial_t\varphi)=0 & {\rm in}\;\; \Omega_R, t>0\\
\partial_t^2\bm u\Big|_{t=0}=\frac{1}{\rho_e}\Delta^*\bm u\Big|_{t=0}=\textbf 0& {\rm in}\;\;  \Omega\\
\partial_t^2\varphi\Big|_{t=0}=c^2\Delta\varphi\Big|_{t=0}=0& {\rm in}\;\;  \Omega_R\\
\bm{\sigma}(\partial_t\bm{u})\n =-\rho_0\partial_t^2\varphi\n& {\rm on}\;\; \Gamma, t>0 \\
\partial_t^2\bm{u}\cdot \n=-\partial_{\n}(\partial_t\varphi)& {\rm on}\;\; \Gamma, t>0\\
\partial_t\varphi=0 & {\rm on}\;\;\Gamma_0, t>0\\
\partial_{\n} (\partial_t\varphi)=\mathscr{T}[\partial_t\varphi]+\partial_t\rho & {\rm on}\;\; \Gamma_R^+, t>0.
\end{cases}
\end{equation}
In order to study (\ref{3.16}), we define another energy function
\ben
\varepsilon_2(t)=e_3(t)+e_4(t)
\enn
with
\begin{align*}
&e_3(t)=\Vert\frac{\rho_0^{1/2}}{c}\partial_t^2\varphi\Vert_{L^2(\Omega_R)}^2+\Vert\rho_0^{1/2}\nabla(\partial_t\varphi)\Vert_{L^2(\Omega_R)}^2,\\
&e_4(t)=\Vert\rho_e^{1/2}\partial_t^2\bm{u}\Vert_{L^2(\Omega)^3}^2+\Vert(\lambda+\mu)^{1/2}\nabla\cdot(\partial_t\bm{u})\Vert_{L^2(\Omega)}^2+\Vert\mu^{1/2}\nabla(\partial_t\bm{u})\Vert_{F(\Omega)}^2.
\end{align*}
Similarly, since
\begin{equation}\label{3.17}
\varepsilon_2(t)-\varepsilon_2(0)=\int_0^t\varepsilon_2^{'}(\tau)d\tau=\int_0^t\Big(e_3^{'}(\tau)+e_4^{'}(\tau)\Big)d\tau,
\end{equation}
we have from the similar steps between (\ref{3.13}) and (\ref{3.15}) with $\varphi$ and $\bm u$ replaced by $\partial_\tau\varphi$ and $\partial_\tau\bm u$, respectively, that
\begin{align}\label{3.18}
\varepsilon_2(t)={}&2Re\rho_0\int_0^t\int_{\Gamma_R^+}\mathscr{T}[\partial_\tau\varphi]\cdot\partial_\tau^2\bar{\varphi}d\gamma d\tau+2Re\rho_0\int_0^t\int_{\Gamma_R^+}\partial_\tau\rho\cdot\partial_\tau^2\bar{\varphi}d\gamma d\tau.\notag\\
\leq{}&2Re\rho_0\int_0^t\int_{\Gamma_R^+}\partial_\tau\rho\cdot\partial_\tau^2\bar{\varphi}d\gamma d\tau\notag\\
={}&2Re\rho_0\Big(\int_{\Gamma_R^+}\partial_\tau\rho\cdot\partial_\tau\bar\varphi\Big|_0^td\gamma-\int_0^t\langle\partial_\tau^2\rho,\partial_\tau\varphi\rangle_{\Gamma_R^+}d\tau\Big)\notag\\
\lesssim{}&\max\limits_{t\in[0,T]}\Vert\partial_t\varphi\Vert_{H^{1}(\Omega_R)}\Big(\max\limits_{t\in[0,T]}\Vert\partial_t\rho\Vert_{H^{-1/2}(\Gamma_R^+)}+\Vert\partial_t^2\rho\Vert_{L^1(0,T;H^{-1/2}(\Gamma_R^+))}\Big)\notag.\\
\leq{}&\epsilon\max\limits_{t\in[0,T]}(\Vert\partial_t\varphi\Vert_{L^2(\Omega_R)}^2+\Vert\nabla\partial_t\varphi\Vert_{L^2(\Omega_R)^3}^2)\notag\\
+&\frac{1}{4\epsilon}\max\limits_{t\in[0,T]}\Vert\partial_t\rho\Vert_{H^{-1/2}(\Gamma_R^+)}^2+\frac{1}{4\epsilon}\Vert\partial_t^2\rho\Vert_{L^1(0,T;H^{-1/2}(\Gamma_R^+))}^2,
\end{align}
where we have used Lemma \ref{lem3.3} and the $\epsilon-$inequality to obtain (\ref{3.18}).\\
Now, by (\ref{3.15}) and (\ref{3.18}) we arrive at
\begin{align}
&\Vert\partial_t\varphi\Vert_{L^2(\Omega_R)}^2+\Vert\nabla\partial_t\varphi\Vert^2_{L^2(\Omega_R)^3}+\Vert\partial_t\bm{u}\Vert^2_{L^2(\Omega)^3}+\Vert\nabla\cdot\bm{u}\Vert^2_{L^2(\Omega)}+\Vert\nabla\bm{u}\Vert^2_{F(\Omega)}\notag\\
\lesssim{}&\varepsilon_1(t)+\varepsilon_2(t)\notag\\
\lesssim{}&2\epsilon\max\limits_{t\in[0,T]}(\Vert\partial_t\varphi\Vert_{L^2(\Omega_R)}^2+\Vert\nabla\partial_t\varphi\Vert_{L^2(\Omega_R)^3}^2)\notag\\
+&\frac{1}{4\epsilon}\Vert\rho\Vert^2_{L^1(0,T;H^{-1/2}(\Gamma_R^+))}+\frac{1}{4\epsilon}\max\limits_{t\in[0,T]}\Vert\partial_t\rho\Vert_{H^{-1/2}(\Gamma_R^+)}^2+\frac{1}{4\epsilon}\Vert\partial_t^2\rho\Vert_{L^1(0,T;H^{-1/2}(\Gamma_R^+))}^2.\notag
\end{align}
Finally, we choose $\epsilon>0$ small enough such that $2\epsilon<1/2$ and apply Cauchy-Schwartz inequality to  obtain
\begin{align}
&\max\limits_{t\in[0,T]}\left[(\Vert\partial_t\varphi\Vert_{L^2(\Omega_R)}+\Vert\nabla\partial_t\varphi\Vert_{L^2(\Omega_R)^3})+(\Vert\partial_t\bm{u}\Vert_{L^2(\Omega)^3}+\Vert\nabla\cdot\bm{u}\Vert_{L^2(\Omega)}+\Vert\nabla\bm{u}\Vert_{F(\Omega)})\right]\notag\\
\lesssim{}&\Vert\rho\Vert_{L^1(0,T;H^{-1/2}(\Gamma_R^+))}+\max\limits_{t\in[0,T]}\Vert\partial_t\rho\Vert_{H^{-1/2}(\Gamma_R^+)}+\Vert\partial_t^2\rho\Vert_{L^1(0,T;H^{-1/2}(\Gamma_R^+))},\notag
\end{align}
which completes the proof.
\end{proof}

\subsection{A priori estimate}
Motivated by \cite{LWW15} or \cite{Bao2018}, we will study in this subsection the fluid-solid interaction problem (\ref{2.8}) in a direct way. The goal is to derive an a priori stability estimate for both the acoustic field $\varphi$ and elastic displacement $\bm{u}$ with an explicit dependence on the time variable.

Recalling that the reduced system (\ref{2.8}),
 for each $t>0$ and $(\psi,\bm{v})\in H=\tilde{H_0^1}(\Omega_R)\times H^1(\Omega)^3$,  we easily have
$$\int_\Omega(\rho_e\frac{\partial^2\bm{u}}{\partial t^2}\cdot\bar{\bm{v}}-\Delta^*\bm u\cdot\bar{\bm v})dx+\rho_0\int_{\Omega_R}\Big(\frac{1}{c^2}\frac{\partial^2\varphi}{\partial t^2}-\Delta\varphi\Big)\bar{\psi}dx=0.$$
By Betti's formula in elastic field and Green's theorem in acoustic field with the transmission conditions, we conclude the variational problem of (\ref{2.8}) in the time domain which is to find $(\varphi,\bm{u})\in H,\ \forall\;t>0,$ such that
\begin{align}\label{3.22}
&\int_\Omega\rho_e\frac{\partial^2\bm{u}}{\partial t^2}\cdot\bar{\bm{v}}dx+\int_{\Omega_R}\frac{\rho_0}{c^2}\frac{\partial^2\varphi}{\partial t^2}\cdot\bar{\psi}dx\notag\\
={}&-\int_\Omega\Big[\mu(\nabla\bm{u}:\nabla \bar{\bm{v}})+(\lambda+\mu)(\nabla\cdot \bm{u})(\nabla\cdot \bar{\bm{v}})\Big]dx-\rho_0\int_{\Omega_R}\nabla\varphi\cdot\nabla\bar{\psi}dx\notag\\
&+\rho_0\int_{\Gamma_R^+}\mathscr{T}[\varphi]\cdot\bar{\psi}d\gamma+\rho_0\int_{\Gamma_R^+}\rho\cdot\bar{\psi}d\gamma+\rho_0\int_{\Gamma}(\frac{\partial\bm{u}}{\partial t}\cdot\n\bar{\psi}-\frac{\partial \varphi}{\partial t}\n\cdot\bar{\bm{v}})d\gamma.
\end{align}
To show the stability of the solution of (\ref{3.22}), the following two lemmas play an important role in the subsequent analysis.
\begin{lem}\label{lem3.5}
Given $\xi\geq 0$ and $\varphi(\cdot,t)\in L^2(0,\xi;H^{1/2}(\Gamma_R^+))$, it holds that
$$Re\int_{\Gamma_R^+}\int_0^\xi\Big(\int_0^t\mathscr{T}[\varphi](x,\tau)d\tau\Big)\bar{\varphi}(x,t)dtd\gamma\leq 0.$$
\end{lem}
\begin{proof}
First, we extend $\varphi$ by 0 with respect to $t$ in the interval $[0,\xi]$, also referring to it as $\varphi$. Following (A.3), Lemma A.1, and Lemma \ref{lem2.1}, we obtain 
\begin{align*}
&Re\int_{\Gamma_R^+}\int_0^{\infty}e^{-2s_1t}\Big(\int_0^t\mathscr{T}[\varphi](x,\tau)d\tau\Big)\bar{\varphi}(x,t)dtd\gamma\\
={}&\frac{1}{2\pi}Re\int_{-\infty}^{\infty}\int_{\Gamma_R^+}s^{-1}\mathscr{B}\circ\mathscr{L}(\varphi)\cdot\mathscr{L}(\bar{\varphi})(s)d\gamma ds_2\\
={}&\frac{1}{2\pi}\int_{-\infty}^{\infty} Re\langle s^{-1}\mathscr{B}[\check{\varphi}],\check{\varphi}\rangle_{\Gamma_R^+}ds_2\\
\leq{}&0,
\end{align*}
which completes the proof by taking $s_1\rightarrow 0.$
\end{proof}
\begin{lem}\label{lem3.6}
Given $\xi\geq 0$ and $\varphi(\cdot,t)\in L^2(0,\xi;H^{1/2}(\Gamma_R^+))$ with $\varphi(\cdot,0)=0$, it holds that
$$Re\int_{\Gamma_R^+}\int_0^\xi\Big(\int_0^t\mathscr{T}[\partial_{\tau}{\varphi}](x,\tau)d\tau\Big)\partial_{\tau}\bar{\varphi}(x,t)dtd\gamma\leq 0.$$
\end{lem}
\begin{proof}
Since the proof is similar to that in Lemma \ref{lem3.5} with $\varphi$ replaced by $\partial_\tau\varphi$,
so we here omit its detailed proof.
\end{proof}
\begin{theorem}
 Let $(\varphi,\bm{u})$ be the solution of (\ref{3.22}). Under the assumption of (\ref{3.11}), it holds for any $T>0$ that
 \begin{align}\label{3.23}
&\Vert\varphi\Vert_{L^{\infty}(0,T;L^2(\Omega_R))}+\Vert\nabla\varphi\Vert_{L^{\infty}(0,T;L^2(\Omega_R)^3)}+\Vert\partial_t\varphi\Vert_{L^{\infty}(0,T;L^2(\Omega_R))}\notag\\
&+\Vert\bm{u}\Vert_{L^{\infty}(0,T;L^2(\Omega)^3)}+\Vert\partial_t\bm{u}\Vert_{L^{\infty}(0,T;L^2(\Omega)^3)}+\Vert\nabla\bm{u}\Vert_{L^{\infty}(0,T;F(\Omega)}+\Vert\nabla\cdot\bm{u}\Vert_{L^{\infty}(0,T;L^2(\Omega))}\notag\\
\lesssim{}&T\Vert\rho\Vert_{L^1(0,T;H^{-1/2}(\Gamma_R^+))}+\Vert\partial_t\rho\Vert_{L^1(0,T;H^{-1/2}(\Gamma_R^+))},
\end{align}
and
\begin{align}\label{3.24}
&\Vert\varphi\Vert_{L^{2}(0,T;L^2(\Omega_R))}+\Vert\nabla\varphi\Vert_{L^{2}(0,T;L^2(\Omega_R)^3)}+\Vert\partial_t\varphi\Vert_{L^{2}(0,T;L^2(\Omega_R))}\notag\\
&+\Vert\bm{u}\Vert_{L^{2}(0,T;L^2(\Omega)^3)}+\Vert\partial_t\bm{u}\Vert_{L^{2}(0,T;L^2(\Omega)^3)}+\Vert\nabla\bm{u}\Vert_{L^{2}(0,T;F(\Omega)}+\Vert\nabla\cdot\bm{u}\Vert_{L^{2}(0,T;L^2(\Omega))}\notag\\
\lesssim{}&T^{\frac{3}{2}}\Vert\rho\Vert_{L^1(0,T;H^{-1/2}(\Gamma_R^+))}+T^{\frac{1}{2}}\Vert\partial_t\rho\Vert_{L^1(0,T;H^{-1/2}(\Gamma_R^+))}.
\end{align}
\end{theorem}

\begin{proof}
For $0<\xi<T$, we introduce an auxiliary function
$$\Psi_1(x,t)=\int_t^\xi\varphi(x,\tau)d\tau,\quad x\in\Omega_R,0\leq t\leq\xi.$$
It can be easily verified that
\begin{align}\label{3.25}
\Psi_1(x,\xi)=0,\;\partial_t\Psi_1(x,t)=-\varphi(x,t).
\end{align}
For any $\phi(x,t)\in L^2\left(0,\xi;L^2(\Omega_R)\right)$, using integration by parts and condition (\ref{3.25}), we obtain
\begin{align}\label{3.26}
&\int_0^\xi\phi(x,t)\cdot\bar{\Psi}_1(x,t)dt\notag\\
={}&\int_0^\xi\int_t^\xi\bar{\varphi}(x,\tau)d\tau\cdot d\Big(\int_0^t\phi(x,\tau)d\tau\Big)dt\notag\\
={}&\int_0^\xi\int_0^t\phi(x,\tau)d\tau\cdot\bar{\varphi}(x,t)dt.
\end{align}
With the aid of (\ref{3.25}), we take the test function $\psi=\Psi_1$ in (\ref{3.22}) to have
\begin{align}\label{3.27}
&Re\frac{\rho_0}{c^2}\int_0^\xi\int_{\Omega_R}\frac{\partial^2\varphi}{\partial t^2}\cdot\bar{\Psi}_1dxdt\notag\\
={}&Re\frac{\rho_0}{c^2}\int_{\Omega_R}\int_0^\xi\Big(\partial_t(\partial_t\varphi\cdot\bar{\Psi}_1)+\partial_t\varphi\cdot\bar{\varphi}\Big)dtdx\notag\\
={}&\frac{\rho_0}{2c^2}\Vert\varphi(\cdot,\xi)\Vert_{L^2(\Omega_R)}^2.
\end{align}
By (\ref{3.26}), we also have
\begin{align*}
&Re\rho_0\int_0^\xi\int_{\Omega_R}\nabla\varphi\cdot\nabla\bar{\Psi}_1dxdt\\
={}&Re\rho_0\int_{\Omega_R}\int_0^\xi\nabla\varphi\cdot\int_t^\xi\nabla\bar{\varphi}(x,\tau)d\tau dt dx\\
={}&\rho_0\int_{\Omega_R}|\int_0^\xi\nabla\varphi(x,t)dt|^2dx-Re\rho_0\int_0^\xi\int_{\Omega_R}\nabla\varphi\cdot\nabla\bar{\Psi}_1dxdt,
\end{align*}
whence
\begin{align}\label{3.29}
Re\rho_0\int_0^\xi\int_{\Omega_R}\nabla\varphi\cdot\nabla\bar{\Psi}_1dxdt=\frac{\rho_0}{2}\int_{\Omega_R}|\int_0^\xi\nabla\varphi(x,t)dt|^2dx
\end{align}
follows.

Similarly, we now define another auxiliary function
$$\Psi_2(x,t)=\int_t^\xi\bm{u}(x,\tau)d\tau,\quad x\in\Omega,0\leq t\leq\xi$$
by $\bm{u}(x,\tau)$.
Clearly, we also have
\begin{align}\label{3.30}
\Psi_2(x,\xi)=0,\partial_t\Psi_2(x,t)=-\bm{u}(x,t).
\end{align}
For any vector $\bm{\omega}(x,t)\in L^2\left(0,\xi;L^2(\Omega)^3\right)$, we have
\begin{align}\label{3.31}
\int_0^\xi\bm{\omega}(x,t)\cdot\bar{\Psi}_2(x,t)dt=\int_0^\xi\int_0^t\bm{\omega}(x,\tau)d\tau\cdot\bar{\bm{u}}(x,t)dt,
\end{align}
which can be proved in a similar way to (\ref{3.26}).

Next, with the aid of (\ref{3.30}), we take the test function $\bm{v}=\Psi_2$ to get
\begin{align}\label{3.32}
Re\rho_e\int_0^\xi\int_{\Omega}\frac{\partial^2\bm{u}}{\partial t^2}\cdot\bar{\Psi}_2dxdt=\frac{\rho_e}{2}\Vert\bm{u}(\cdot,\xi)\Vert_{L^2(\Omega)^3}^2,
\end{align}
and
\begin{equation*}\label{3.33}
\begin{aligned}
&Re\int_0^\xi\int_{\Omega}\Big[\mu(\nabla\bm{u}:\nabla\bar{\Psi}_2)+(\lambda+\mu)(\nabla\cdot\bm{u})(\nabla\cdot\bar{\Psi}_2)\Big]dxdt\\
={}&Re\int_0^\xi\int_{\Omega}\Big[\mu(\nabla\bm{u}:\int_t^\xi\nabla\bar{\bm{u}}d\tau)+(\lambda+\mu)(\nabla\cdot\bm{u})(\int_t^\xi\nabla\cdot\bar{\bm{u}}d\tau)\Big]dxdt\\
={}&Re\mu\int_{\Omega}\int_0^\xi\nabla\bm{u}:\int_0^\xi\nabla\bar{\bm{u}}d\tau dtdx-Re\mu\int_{\Omega}\int_0^\xi\nabla\bm{u}:\int_0^t\nabla\bar{\bm{u}}d\tau dtdx\\
&+Re(\lambda+\mu)\int_{\Omega}\int_0^\xi(\nabla\cdot\bm{u})(\int_0^\xi\nabla\cdot\bar{\bm{u}}d\tau)dtdx\\
&-Re(\lambda+\mu)\int_{\Omega}\int_0^\xi(\nabla\cdot\bm{u})(\int_0^t\nabla\cdot\bar{\bm{u}}d\tau)dtdx,
\end{aligned}
\end{equation*}
using (\ref{3.31}), we can get
\begin{align}\label{3.34}
&Re\int_0^\xi\int_{\Omega}\Big[\mu(\nabla\bm{u}:\nabla\bar{\Psi}_2)+(\lambda+\mu)(\nabla\cdot\bm{u})(\nabla\cdot\bar{\Psi}_2)\Big]dxdt\notag\\
={}&\frac{1}{2}\int_{\Omega}\Big[\mu(\int_0^\xi\nabla\bm{u}dt:\int_0^\xi\nabla\bar{\bm{u}}dt)+(\lambda+\mu)|\int_0^\xi\nabla\cdot\bm{u}dt|^2\Big]dx.
\end{align}
Moreover, by simple calculations, we have
\begin{align}\label{3.35}
&Re\int_0^\xi\int_{\Gamma}\frac{\partial\bm{u}}{\partial t}\cdot\n\bar{\Psi}_1d\gamma dt-Re\int_0^\xi\int_{\Gamma}\frac{\partial \varphi}{\partial t}\n\cdot\bar{\bm{v}}(x,\tau)d\gamma dt\notag\\
={}&Re\int_{\Gamma}\int_0^\xi\frac{\partial\bm{u}}{\partial t}\cdot\n\int_t^\xi\bar{\varphi}(x,\tau)d\tau dt d\gamma -Re\int_{\Gamma}\int_0^\xi\frac{\partial \varphi}{\partial t}\n\cdot\int_t^\xi\bar{\bm{u}}(x,\tau)d\tau dt d\gamma\notag\\
={}&Re\int_{\Gamma}\int_0^\xi\bm{u}\cdot\n\bar{\varphi}(x,t)dt d\gamma -Re\int_{\Gamma}\int_0^\xi\varphi\n\cdot\bar{\bm{u}}(x,t) dt d\gamma\notag\\
={}&0.
\end{align}
Integrating (\ref{3.22}) from $t=0$ to $t=\xi$ and taking the real parts yields
\begin{align}\label{3.36}
&\frac{\rho_0}{2c^2}\Vert\varphi(\cdot,\xi)\Vert_{L^2(\Omega_R)}^2+\frac{\rho_e}{2}\Vert\bm{u}(\cdot,\xi)\Vert_{L^2(\Omega)^3}^2\notag\\
&+\frac{1}{2}\int_{\Omega}\Big[\mu(\int_0^\xi\nabla\bm{u}dt:\int_0^\xi\nabla\bar{\bm{u}}dt)+(\lambda+\mu)|\int_0^\xi\nabla\cdot\bm{u}dt|^2\Big]dx+\frac{\rho_0}{2}\int_{\Omega_R}|\int_0^\xi\nabla\varphi(x,t)dt|^2dx\notag\\
={}&Re\rho_0\int_0^\xi\int_{\Gamma_R^+}\mathscr{T}[\varphi]\cdot\bar{\Psi}_1d\gamma dt+Re\rho_0\int_0^\xi\int_{\Gamma_R^+}\rho\cdot\bar{\Psi}_1d\gamma dt.
\end{align}
In what follows, we estimate the two terms on the right-hand side of (\ref{3.36}) separately. 
First, using (\ref{3.26}) and Lemma \ref{lem3.5}, it holds that
\begin{align}\label{3.37}
&Re\rho_0\int_0^\xi\int_{\Gamma_R^+}\mathscr{T}[\varphi]\cdot\bar{\Psi}_1d\gamma dt\notag\\
={}&Re\rho_0\int_{\Gamma_R^+}\int_0^\xi\Big(\int_0^t\mathscr{T}[\varphi](x,\tau)d\tau\Big)\cdot\bar{\varphi}(x,t) dt d\gamma\notag \\
\leq{}&0.
\end{align}
For $0\leq t\leq\xi\leq T$, it is seen from (\ref{3.26}) that
\begin{align}\label{3.38}
&Re\int_0^\xi\int_{\Gamma_R^+}\rho\cdot\bar{\Psi}_1d\gamma dt\notag\\
={}&Re\int_0^\xi\Big(\int_0^t\int_{\Gamma_R^+}\rho(\cdot,\tau)d\gamma d\tau\Big)\bar{\varphi}(\cdot,t)dt\notag\\
\lesssim{}&\int_0^\xi\int_0^t\Vert\rho(\cdot,\tau)\Vert_{H^{-1/2}(\Gamma_R^+)}\Vert\varphi(\cdot,t)\Vert_{H^{1}(\Omega_R)}d\tau dt\notag\\
\lesssim{}&\Big(\int_0^\xi\Vert\rho(\cdot,t)\Vert_{H^{-1/2}(\Gamma_R^+)}dt\Big)\cdot\Big(\int_0^\xi\Vert\varphi(\cdot,t)\Vert_{H^{1}(\Omega_R)}dt\Big).
\end{align}
where we have used trace theorem in the second to last inequality.

Substituting (\ref{3.37})-(\ref{3.38}) into (\ref{3.36}), one has for any $\xi\in[0,T]$ that
\begin{align}\label{3.39}
&\frac{\rho_0}{2c^2}\Vert\varphi(\cdot,\xi)\Vert_{L^2(\Omega_R)}^2+\frac{\rho_e}{2}\Vert\bm{u}(\cdot,\xi)\Vert_{L^2(\Omega)^3}^2\notag\\
&+\frac{1}{2}\int_{\Omega}[\mu(\int_0^\xi\nabla\bm{u}dt:\int_0^\xi\nabla\bar{\bm{u}}dt)+(\lambda+\mu)|\int_0^\xi\nabla\cdot\bm{u}dt|^2]dx+\frac{\rho_0}{2}\int_{\Omega_R}|\int_0^\xi\nabla\varphi(x,t)dt|^2dx\notag\\
\lesssim{}&\Big(\int_0^\xi\Vert\rho(\cdot,t)\Vert_{H^{-1/2}(\Gamma_R^+)}dt\Big)\cdot\Big(\int_0^\xi\Vert\varphi(\cdot,t)\Vert_{H^{1}(\Omega_R)}dt\Big).
\end{align}
Noticing that the right-hand side of (\ref{3.39}) contains term
$$\int_0^\xi\Vert\varphi(\cdot,t)\Vert_{H^{1}(\Omega_R)}dt=\int_0^\xi\Big(\int_{\Omega_R}(|\varphi(\cdot,t)|^2+|\nabla\varphi(\cdot,t)|^2)dx\Big)^\frac{1}{2}dt,$$
we shall use the $\epsilon$-inequality to control terms $\Vert\varphi\Vert_{L^2(\Omega_R)}^2$ and $\Vert\nabla\varphi\Vert_{L^2(\Omega_R)}^2$. Consequently, a new reduced system need to be considered:
\begin{equation}\label{3.40}
\begin{cases}
\rho_e\partial_t^2(\partial_t\bm u)-\Delta^*(\partial_t\bm u)=0,& {\rm in}\;\; \Omega, t>0 \\
\partial_t^2(\partial_t\varphi)-c^2\Delta(\partial_t\varphi)=0,& {\rm in}\;\; \Omega_R, t>0\\
\partial_t^2\bm u\Big|_{t=0}=\frac{1}{\rho_e}\Delta^*\bm u\Big|_{t=0}=\textbf 0,& {\rm in}\;\;  \Omega\\
\partial_t^2\varphi\Big|_{t=0}=c^2\Delta\varphi\Big|_{t=0}=0,& {\rm in}\;\;  \Omega_R\\
\bm{\sigma}(\partial_t\bm{u})\n =-\rho_0\partial_t^2\varphi\n,& {\rm on}\;\; \Gamma, t>0 \\
\partial_t^2\bm{u}\cdot \n=-\partial_{\n}(\partial_t\varphi),& {\rm on}\;\; \Gamma, t>0\\
\partial_t\varphi=0,& {\rm on}\;\;\Gamma_0, t>0\\
\partial_{\n} (\partial_t\varphi)=\mathscr{T}[\partial_t\varphi]+\partial_t\rho,& {\rm on}\;\; \Gamma_R^+, t>0.
\end{cases}
\end{equation}
Following the same steps as in deriving (\ref{3.22}), we obtain the variational formulation of (\ref{3.40}):
\begin{align}\label{3.41}
&\int_\Omega\rho_e\partial_t^2(\partial_t\bm{u})\cdot\bar{\bm{v}}dx+\int_{\Omega_R}\frac{\rho_0}{c^2}\partial_t^2(\partial_t\varphi)\cdot\bar{\psi}dx\notag\\
={}&-\int_\Omega\Big[\mu(\nabla(\partial_t\bm{u}):\nabla \bar{\bm{v}})+(\lambda+\mu)(\nabla\cdot (\partial_t\bm{u}))(\nabla\cdot \bar{\bm{v}})\Big]dx
-\rho_0\int_{\Omega_R}\nabla(\partial_t\varphi)\cdot\nabla\bar{\psi}dx\notag\\
&+\rho_0\int_{\Gamma_R^+}\mathscr{T}[\partial_t\varphi]\cdot\bar{\psi}d\gamma
+\rho_0\int_{\Gamma_R^+}\partial_t\rho\cdot\bar{\psi}d\gamma+\rho_0\int_{\Gamma}(\partial_t^2\bm{u}\cdot\n\bar{\psi}-\partial_t^2 \varphi\n\cdot\bar{\bm{v}})d\gamma.
\end{align}
Define two related auxiliary functions
\begin{align*}
&\Psi_3(x,t)=\int_t^\xi\partial_\tau\varphi(x,\tau)d\tau,\quad x\in\Omega_R,0\leq t\leq\xi\leq T,\\
&\Psi_4(x,t)=\int_t^\xi\partial_\tau\bm{u}(x,\tau)d\tau,\quad x\in\Omega,0\leq t\leq\xi\leq T.
\end{align*}
By the integration by parts, it can be similarly concluded that
\begin{align}
&Re\frac{\rho_0}{c^2}\int_0^\xi\int_{\Omega_R}\partial_t^2(\partial_t\varphi)\cdot\bar{\Psi}_3dxdt=\frac{\rho_0}{2c^2}\Vert\partial_t\varphi(\cdot,\xi)\Vert_{L^2(\Omega_R)}^2,\\
&Re\rho_0\int_0^\xi\int_{\Omega_R}\nabla(\partial_t\varphi)\cdot\nabla\bar{\Psi}_3dxdt=\frac{\rho_0}{2}\Vert\nabla\varphi(\cdot,\xi)\Vert_{L^2(\Omega_R)^3}^2,\\
&Re\rho_e\int_0^\xi\int_{\Omega}\partial_t^2(\partial_t\bm{u})\cdot\bar{\Psi}_4dxdt=\frac{\rho_e}{2}\Vert\partial_t\bm{u}(\cdot,\xi)\Vert_{L^2(\Omega)^3}^2,\\
&Re\int_0^\xi\int_{\Omega}\Big[\mu(\nabla\partial_t\bm{u}:\nabla\bar{\Psi}_4)+(\lambda+\mu)(\nabla\cdot\partial_t\bm{u})(\nabla\cdot\bar{\Psi}_4)\Big]dxdt\notag\\
={}&\frac{1}{2}\mu\Vert\nabla\bm{u}(\cdot,\xi)\Vert_{F(\Omega)}^2+\frac{1}{2}(\lambda+\mu)\Vert\nabla\cdot\bm{u}(\cdot,\xi)\Vert_{L^2(\Omega)}^2,\\
&Re\int_0^\xi\int_{\Gamma}\partial_t^2\bm u\cdot\n\bar{\Psi}_3d\gamma dt-Re\int_0^\xi\int_{\Gamma}\partial_t^2\varphi\n\cdot\bar{\Psi}_4(x,\tau)d\gamma dt=0.
\end{align}
Choosing test functions $\psi=\Psi_3$ and $\bm v=\Psi_4$ in (\ref{3.41}), integrating it from $t=0$ to $t=\xi$ and taking the real parts yields
\begin{align}\label{3.46}
&\frac{\rho_0}{2c^2}\Vert\partial_t\varphi(\cdot,\xi)\Vert_{L^2(\Omega_R)}^2+\frac{\rho_0}{2}\Vert\nabla\varphi(\cdot,\xi)\Vert_{L^2(\Omega_R)^3}^2+\frac{\rho_e}{2}\Vert\partial_t\bm{u}(\cdot,\xi)\Vert_{L^2(\Omega)^3}^2\notag\\
&+\frac{1}{2}\mu\Vert\nabla\bm{u}(\cdot,\xi)\Vert_{F(\Omega)}^2+\frac{1}{2}(\lambda+\mu)\Vert\nabla\cdot\bm{u}(\cdot,\xi)\Vert_{L^2(\Omega)}^2\notag\\
={}&Re\rho_0\int_0^\xi\int_{\Gamma_R^+}\mathscr{T}[\partial_t\varphi]\cdot\bar{\Psi}_3d\gamma dt+Re\rho_0\int_0^\xi\int_{\Gamma_R^+}\partial_t\rho\cdot\bar{\Psi}_3d\gamma dt\notag\\
\lesssim{}&Re\int_0^\xi\int_{\Gamma_R^+}\partial_t\rho\cdot\bar{\Psi}_3d\gamma dt\notag\\
={}&Re\int_0^\xi\Big(\int_0^t\int_{\Gamma_R^+}\partial_\tau\rho(\cdot,\tau)d\gamma d\tau\Big)\partial_t\bar{\varphi}(\cdot,t)dt\notag\\
={}&Re\int_{\Gamma_R^+}\int_0^t\partial_\tau\rho(\cdot,\tau)d\tau\cdot\bar{\varphi}(\cdot,t)\Big|_{t=0}^{t=\xi}d\gamma-Re\int_0^\xi\int_{\Gamma_R^+}\partial_t\rho(\cdot,t)\bar{\varphi}(\cdot,t)d\gamma dt\notag\\
\lesssim{}&\int_0^\xi\Vert\partial_t\rho(\cdot,t)\Vert_{H^{-1/2}(\Gamma_R^+)}\Vert\varphi(\cdot,t)\Vert_{H^{1}(\Omega_R)}dt.
\end{align}
where we have used (\ref{3.26}) and Lemma \ref{lem3.6} to make sure
$$Re\rho_0\int_0^\xi\int_{\Gamma_R^+}\mathscr{T}[\partial_t\varphi]\cdot\bar{\Psi}_3d\gamma dt\leq 0,$$
which is similar to (\ref{3.37}).

Combining (\ref{3.39}) and (\ref{3.46}), we have that
\begin{align}\label{3.49}
&\Vert\varphi(\cdot,\xi)\Vert_{L^2(\Omega_R)}^2+\Vert\partial_t\varphi(\cdot,\xi)\Vert_{L^2(\Omega_R)}^2+\Vert\nabla\varphi(\cdot,\xi)\Vert_{L^2(\Omega_R)^3}^2\notag\\
&+\Vert\bm{u}(\cdot,\xi)\Vert_{L^2(\Omega)^3}^2+\Vert\partial_t\bm{u}(\cdot,\xi)\Vert_{L^2(\Omega)^3}^2+\Vert\nabla\bm{u}(\cdot,\xi)\Vert_{F(\Omega)}^2+\Vert\nabla\cdot\bm{u}(\cdot,\xi)\Vert_{L^2(\Omega)}^2\notag\\
\lesssim{}&\Big(\int_0^\xi\Vert\rho(\cdot,t)\Vert_{H^{-1/2}(\Gamma_R^+)}dt\Big)\cdot\Big(\int_0^\xi\Vert\varphi(\cdot,t)\Vert_{H^{1}(\Omega_R)}dt\Big)\notag\\
&+\int_0^\xi\Vert\partial_t\rho(\cdot,t)\Vert_{H^{-1/2}(\Gamma_R^+)}\Vert\varphi(\cdot,t)\Vert_{H^{1}(\Omega_R)}dt.
\end{align}
Taking the $L^{\infty}$-norm in both sides of (\ref{3.49}) with respect to $\xi$, and using the $\epsilon$-inequality, we obtain
\begin{align}\label{3.51}
&\Vert\varphi\Vert_{L^{\infty}(0,T;L^2(\Omega_R))}^2+\Vert\partial_t\varphi\Vert_{L^{\infty}(0,T;L^2(\Omega_R))}^2+\Vert\nabla\varphi\Vert_{L^{\infty}(0,T;L^2(\Omega_R)^3)}^2\notag\\
&+\Vert\bm{u}\Vert_{L^{\infty}(0,T;L^2(\Omega)^3)}^2+\Vert\partial_t\bm{u}\Vert_{L^{\infty}(0,T;L^2(\Omega)^3)}^2+\Vert\nabla\bm{u}\Vert_{L^{\infty}(0,T;F(\Omega))}^2+\Vert\nabla\cdot\bm{u}\Vert_{L^{\infty}(0,T;L^2(\Omega))}^2\notag\\
\lesssim{}&T^2\Vert\rho\Vert_{L^{1}(0,T;H^{-1/2}(\Gamma_R^+))}^2+\Vert\partial_t\rho\Vert_{L^{1}(0,T;H^{-1/2}(\Gamma_R^+))}^2,
\end{align}
which implies estimate (\ref{3.23}) after applying Cauchy-Schwartz inequality.

Integrating (\ref{3.49}) with respect to $\xi$ from $0$ to $T$ and using Cauchy-Schwartz inequality, we have
\begin{align}\label{3.53}
&\Vert\varphi\Vert_{L^{2}(0,T;L^2(\Omega_R))}^2+\Vert\partial_t\varphi\Vert_{L^{2}(0,T;L^2(\Omega_R))}^2+\Vert\nabla\varphi\Vert_{L^{2}(0,T;L^2(\Omega_R)^3)}^2\notag\\
&+\Vert\bm{u}\Vert_{L^{2}(0,T;L^2(\Omega)^3)}^2+\Vert\partial_t\bm{u}\Vert_{L^{2}(0,T;L^2(\Omega)^3)}^2+\Vert\nabla\bm{u}\Vert_{L^{2}(0,T;F(\Omega))}^2+\Vert\nabla\cdot\bm{u}\Vert_{L^{2}(0,T;L^2(\Omega))}^2\notag\\
\lesssim{}&T^3\Vert\rho\Vert_{L^{1}(0,T;H^{-1/2}(\Gamma_R^+))}+T\Vert\partial_t\rho\Vert_{L^{1}(0,T;H^{-1/2}(\Gamma_R^+))},
\end{align}
where we have used $\epsilon$-inequality again in deriving (\ref{3.53}). This implies estimate (\ref{3.24}) by applying the Cauchy-Schwartz inequality.
\end{proof}

\begin{remark}\label{rm3.8}{\rm
By similar discussions, main results obtained in this paper can be easily extended to the time-dependent fluid-solid interaction problem with the case of an incident point source wave as well as other boundary
condition such as the Neumann boundary condition.
 }
\end{remark}

\appendix
\section{Laplace transform}\label{ap1}
For each $s\in\C_+$, the Laplace transform of the vector field $\bm{u}(t)$ is defined as:
$$\check{\bm{u}}(s)=\mathscr{L}(\bm{u})(s)=\int_0^{\infty}e^{-st}\bm{u}(t)dt.$$
Some related properties on the Laplace transform and its inversion are summarized as
\begin{align}\label{A.1}
\mathscr{L}(\frac{d\bm{u}}{dt})(s)=s\mathscr{L}(\bm{u})(s)-\bm{u}(0),\tag{A.1}\\ \label{A.2}
\mathscr{L}(\frac{d^2\bm{u}}{dt^2})(s)=s^2\mathscr{L}(\bm{u})(s)-s\bm{u}(0)-\frac{d\bm{u}}{dt}(0),\tag{A.2}\\ \label{A.3}
\int_0^t\bm{u}(\tau)d\tau=\mathscr{L}^{-1}(s^{-1}\check{\bm{u}})(s),\tag{A.3}
\end{align}
which can be easily verified from the integration by parts.

Next, we present the relation between Laplace and Fourier transform.
According to the definition on the Fourier transform, it holds
$$\mathscr{F}(\bm{u}(\cdot)e^{-s_1\cdot})=\int_{-\infty}^{+\infty}\bm{u}(t)e^{-s_1t}e^{-is_2t}dt=\int_{0}^{\infty}\bm{u}(t)e^{-(s_1+is_2)t}dt=\mathscr{L}(\bm{u})(s_1+is_2).$$
We can verify from the formula of the inverse Fourier transform that
$$\bm{u}(t)e^{-s_1t}=\mathscr{F}^{-1}\{\mathscr{F}(\bm{u}(\cdot)e^{-s_1\cdot})\}=\mathscr{F}^{-1}\Big(\mathscr{L}(\bm{u}(s_1+is_2))\Big),$$
which implies that
\begin{align}
    \bm{u}(t)=\mathscr{F}^{-1}\Big(e^{s_1t}\mathscr{L}(\bm{u}(s_1+is_2))\Big),\tag{A.4}\label{A.4}
\end{align}
where $\mathscr{F}^{-1}$ denotes the inverse Fourier transform with respect to $s_2.$
\begin{lem}[Parseval identity]
    If $\check{\bm{u}}=\mathscr{L}(\bm{u})$ and $\check{\bm{v}}=\mathscr{L}(\bm{v})$, then
    \begin{align}
        \frac{1}{2\pi}\int_{-\infty}^{\infty}\check{\bm{u}}(s)\cdot\check{\bm{v}}(s)ds_2=\int_0^{\infty}e^{-2s_1t}\bm{u}(t)\cdot\bm{v}(t)dt.\tag{A.5}\label{A.5}
    \end{align}
    \noindent for all $s_1>\lambda$ where $\lambda$ is the abscissa of convergence for the Laplace transform of $\bm{u}$ and $\bm{v}$.
\end{lem}
\begin{lem}{\rm(cf.\cite[Theorem 43.1]{treves1975})}
    Let $\check{\bm{\omega}}(s)$ denotes a holomorphic function in the half plane $s_1>\sigma_0$, valued in the Banach space $\E$. The following statements are equivalent:
    \begin{enumerate}
      \item[1)] there is a distribution $\omega\in\mathcal{D}_+^{'}(\E)$ whose Laplace transform is equal to $\check{\bm{\omega}}(s)$;
      \item[2)] there is a $\sigma_1$ with $\sigma_0\leq\sigma_1<\infty$ and an integer $m\geq0$ such that for all complex numbers $s$ with $s_1>\sigma_1$ , it holds that $\Vert\check{\bm{\omega}}(s)\Vert_{\E} \lesssim(1+|s|)^m,$ where $\mathcal{D}_+^{'}(\E)$ is the space of distributions on the real line which vanish identically in the open negative half line.
    \end{enumerate}
\end{lem}

\section*{Acknowledgements}

This work was supported by NNSF of China under Grant No. 11771349, by ``the Fundamental Research Funds for the Central Universities" under Grant No. 1191329813, by the China Postdoctoral Science Foundation under Grants No. 2015M580827 and No. 2016T90900, and by Postdoctoral research project of Shaanxi Province of China under Grant No. 2016BSHYDZZ52.


\end{document}